\documentclass[12pt, reqno]{amsart}
     \makeatletter
     \def\section{\@startsection{section}{1}%
     \z@{.7\linespacing\@plus\linespacing}{.5\linespacing}%
     {\bfseries
     \centering
     }}
     \def\@secnumfont{\bfseries}
     \makeatother
\usepackage{amssymb}
\usepackage[final]{hyperref}

\usepackage{xcolor}

\usepackage{amsrefs}

\usepackage{comment}

 \usepackage{wrapfig}

\usepackage{enumerate}

\theoremstyle{plain}
\newtheorem{theorem}{Theorem}[section]
\newtheorem{lemma}[theorem]{Lemma}
\newtheorem{prop}[theorem]{Proposition}
\newtheorem{corollary}[theorem]{Corollary}

\theoremstyle{definition}

\newtheorem{remark}[theorem]{Remark}
\newtheorem{example}[theorem]{Example} 
 
\numberwithin{equation}{section}

\usepackage{amssymb}
\usepackage[final]{hyperref}
 
 \usepackage{wrapfig}

\newcommand{\re}{\operatorname {Re}}
\newcommand{\im}{\operatorname {Im}}

\newcommand{\mcd}{{\mathcal D}}

\newcommand{\mcz}{{\mathcal Z}}

\newcommand{\mch}{{\mathcal H}}
\newcommand{\mbr}{{\mathbb R}}

\newcommand{\mbc}{{\mathbb C}}

\newcommand{\mcp}{{\mathcal P}}

\newcommand{\mbmn}{{M \cdot M}}
\newcommand{\mbx}{{X \cdot X}}

\newcommand{\mbq}{{\mathbb Q}}

\newcommand{\mbz}{{\mathbb Z}}

\newcommand{\ovs}{\overline{\sigma}}

\newcommand{\rank}{\operatorname {rank}}
\newcommand{\Ort}{\operatorname {O}}

\begin{document}

\title[Rotational Symmetries in Polynomial Rings]{Rotational Symmetries in Polynomial Rings}

 \author{Keith Conrad}
 \address{Department of Mathematics \\
  University  of Connecticut\\
Storrs, CT 062569\\
e-mail: \sl kconrad@math.uconn.edu} 

\author{Ambar N.~Sengupta}
\address{Department of Mathematics \\
  University  of Connecticut\\
Storrs, CT 062569\\
e-mail: \sl ambarnsg@gmail.com} 



\dedicatory{}

\begin{abstract}  
We  obtain results describing the behavior of the action of rotation generators on  polynomials over a commutative ring.    We also explore harmonic polynomials in a purely algebraic setting.
\end{abstract}

\maketitle
  
 \section{Introduction and summary of results}
 
 The purpose of this paper is to study the action of rotation generators on a polynomial ring. We establish numerous algebraic results, many of which have classical analytic counterparts for functions on Euclidean space.

 Let $R$ be a ring  (always assumed commutative with identity) and consider the polynomial ring $\mcp_N= R[X_1,\ldots, X_N]$ with $N \geq 2$.  Classically $R$ is  $\mbr$ or $\mbc$. For us, $R$ may be a ring of polynomials over further indeterminates with coefficients in some commutative ring.  Thus, for our purposes, it is necessary to build the framework of results with $R$ being a ring rather than a field. Other examples  for $R$ of interest include the ring of $p$-adic integers and the ring of formal power series over some integral domain.

 By a {\em rotation generator} we mean an operator on $\mcp_N$ of the form
\begin{equation}\label{E:defMjk}
M_{jk}=X_j\partial_k - X_k\partial_j,
\end{equation}
where $j, k\in\{1,\ldots, N\}$ with $j \not= k$ and $\partial_j = \partial/\partial X_j$. 
The $R$-linear span of these operators is a Lie algebra over $R$. 
The skew-symmetry 
$$
M_{jk}=-M_{kj},
$$
implies $M_{jk}^2$ is independent of the ordering of $j$ and $k$.

We also work with the Laplacian operator
\begin{equation}\label{E:defDeltaN}
\Delta_N=\partial_1^2+\cdots +\partial_N^2
\end{equation}
and sometimes we will drop the subscript $N$ from $\Delta_N$.
The ``quadratic Casimir'' operator  ${\mbmn}$ is given  by
\begin{equation}\label{E:MdotM}
\mbmn  =\sum_{\{j,k\}\in P_2(N)}M_{jk}^2,
\end{equation}
where $P_2(N)$ is the set of all $2$-element subsets of $\{1,\ldots, N\}$. 
Lengthy but straightforward computations produce the following commutator relations: on $\mcp_N$, 
\begin{equation}\label{E:MjkM2comm}
\begin{split}
 [M_{jk}, M_{lm}]&=0 \qquad\hbox{if $\{j,k\}, \{l,m\} \in P_2(N)$ are disjoint,}\\
 [M_{jk}, M_{kl}] &=M_{jl} \qquad\hbox{\!\!\!\!\!\!\!\!\! if $\{j,k\}, \{k,l\} \in P_2(N)$ with $j \not= l$ },\\
 [M_{jk}, \mbmn]&=0 \qquad\hbox{if $\{j,k\} \in P_2(N)$}.
\end{split}
\end{equation}

\subsection{Summary of results}\label{ss:sumr}
We present a sample of our results, indicating their overall narrative role in the development.  
The version of the results stated below 
may have a weaker hypothesis on $R$ when it is proved later.

The following result, which is established later (with slightly weaker hypotheses on $R$) 
as Propositions \ref{P:Mxyp0N} and \ref{P:Mquad2}, 
is the algebraic counterpart of the geometric result that if a function on $\mbr^N$ is invariant under all rotations around the origin then it is a function of the Euclidean norm. 

\begin{prop}\label{P:Mquad} 
Suppose the ring $R$ contains $\mbq$.
A polynomial $p\in R[X_1,\ldots, X_N]$ is annihilated by all the rotation generators 
$M_{jk}$ if and only if $p$ is a polynomial in ${\mbx} = X_1^2+\cdots +X_N^2$ with $R$-coefficients. 
An $R$-linear derivation on $R[X_1,\ldots, X_N]$   annihilates ${\mbx}$ if and only if it is an $R[X_1,\ldots,X_N]$-linear combination 
of the operators $M_{jk}$.
\end{prop}

For a more general picture, let $D_A$ be the set of all derivations on an $R$-algebra $A$; then $D_A$ is a Lie algebra over $A$. If $D$ is any Lie subalgebra of $D_A$ then the set 
$$N(D)\stackrel{\rm def}{=}\bigcap_{L\in D}\ker L$$
of common zeros of all $L\in D$ 
 is a subalgebra of $A$. Conversely for any subalgebra $B\subset A$ the annihilator 
 $${\mathcal A}(B)\stackrel{\rm def}{=}\{L\in D_A: B\subset \ker L\}$$
  is a Lie subalgebra of $D_A$. In this framework, Proposition \ref{P:Mquad}  is the case where $A=R[X_1,\ldots, X_N]$, and  
  $B = R[\mbx]$ is the subalgebra generated over $R$ by the quadratic form ${\mbx}$: it says 
\begin{quote}{ \em The  derivations on $R[X_1,\ldots, X_N]$ that annihilate  ${\mbx}$ form the Lie algebra  spanned by all 
the operators $M_{jk}$  with coefficients in $R[X_1,\ldots,X_N]$, 
and every polynomial in $R[X_1,\ldots,X_N]$ annihilated by that Lie algebra is in the subalgebra $B$.}
\end{quote}

We summarize the major points of the paper, mentioning some analytic counterparts of the algebraic results.

\begin{enumerate}
    \item We use a purely algebraic approach, working mostly with a polynomial ring $R[X_1,\ldots, X_N]$ and we assume 
    that $N\geq 2$. A few results need $N>2$.  Whether results hold for $N = 1$ can be checked by the reader in each case. 
    \item {\bf Summary of notation}. 
    For simplicity we will assume in this section that $R$ is a ring containing $\mbq$, such as any field of characteristic 0 or a polynomial ring over $\mbq$.  
    The $R$-algebra $\mcp_N=R[X_1,\ldots, X_N]$ is the direct sum 
    \begin{equation}\label{PPsum}
    \mcp_N=\bigoplus_{d\in\mbz}\mcp_{N,d},
    \end{equation}
    where $\mcp_{N,d}$ is the $R$-submodule of $\mcp_N$ 
    consisting of all homogeneous degree-$d$ polynomials together with $0$. Set $\mcp_{N,d} = \{0\}$ if $d<0$. The Laplacian operator 
    on $\mcp_N$ is 
    $$\Delta_N=\sum_{j=1}^N\partial_j^2:\mcp_{N}\to\mcp_{N-2},$$
    and the restriction of $\Delta_N$ to $\mcp_{N,d}$ is $\Delta_{N,d}:\mcp_{N,d}\to\mcp_{N,d-2}$. A polynomial is {\em harmonic} if it is in $\ker\Delta_N$;
    the space  of harmonic polynomials is 
    $$\mch_N=\ker\Delta_N=\bigoplus_{d\in\mbz}\mch_{N,d},$$
    where $\mch_{N,d}=\ker\Delta_{N,d}$. We work with the {\em  rotation generators}
    $$M_{jk}=X_j\partial_k-X_k\partial_j,$$
    which map $\mcp_{N,d}$ to $\mcp_{N,d}$, and the {\em quadratic Casimir}
    $${\mbmn}=\sum_{\{j,k\}\in P_2(N)}M_{jk}^2,$$
    where $P_2(N)$ is the $2$-element subsets $\{j,k\}$ of $\{1,2,\ldots, N\}$.
    Lastly, 
    $${\mbx}=X_1^2+\cdots+X_N^2,$$
    and for each $c\in R$ we set 
    $$\mcz_N(c)=\hbox{ideal in $R[X_1,\ldots,X_N]$ generated by $\mbx-c$.}$$
    
    \item (Prop. \ref{P:laplace}) We often use the following operator identity on $\mcp_N$: 
   $$
      (\mbx)\Delta_N= (r\partial_r)^2+(N-2)r\partial_r+ \mbmn,
    $$
    where we define the {\it Euler operator}  
    $$
    r\partial_r=\sum_{j=1}^NX_j\partial_j. 
    $$
 
    \item (Remark \ref{rk1}) {\em If $S$ is  a subset of $ \{1,\ldots, N\}$ with at least $2$ elements and 
   $p\in \mcp_N$ is in $\ker M_{jk}$ for all distinct $j, k\in S$, then $p$ is a polynomial in $\sum_{j\in S}X_j^2$ with coefficients in 
   $R[X_l; l \notin S]$}. Analytic counterpart: if a  function $p$ on $\mbr^n$ is invariant under all rotations in the coordinates $x_1,\ldots, x_k$, then $p$  is a function of $x_1^2+\cdots+x_k^2$ and the remaining coordinate variables.
    
    \item  (Prop. \ref{P:rotinvMjk}(iii)) {\em If a nonzero homogeneous polynomial $p\in\mcp_{N,d}$ is such that $M_{jk}p\in\mcz_N(c)$ for some $c \in R$ that is not a zero divisor and all distinct $j, k\in\{1,\ldots, N\}$ then $d$ is even and $p$ is an $R$-multiple of  $(\mbx)^{d/2}$.} Analytic counterpart: if a homogeneous function $p$ on $\mbr^N$ restricted to some sphere centered at the origin is rotation invariant then $p$ is a function of the norm on $\mbr^N$.  
        
  \item  (Prop. \ref{P:kerDelt2}) {\em The only harmonic polynomial in $\mcz_N(c)$ is $0$.} Analytic counterpart: a harmonic function on $\mbr^N$ that vanishes on a sphere centered at the origin is $0$.
    
    \item  (Prop. \ref{P:harmdecomp1})  {\em Every polynomial $p\in R[X_1,\ldots, X_N]$ can be expressed uniquely in the form
    $$p=p_0+(\mbx)p_1+\cdots +(\mbx)^sp_{s}$$
    where each $p_j$ is harmonic.  If $p\in\mcp_{N,d}$ $($homogeneous of degree $d$$)$ then each $p_j$ is also homogeneous, with $p_j\in\mch_{N,d-2j}$.}
    
      \item   (Prop. \ref{P:defLc})  {\em For each $c\in R$ and $p\in\mcp_N$ there is a {\em unique harmonic polynomial} $p_*\in\mch_N$ such that $p-p_*\in\mcz_N(c)$. } Analytic counterpart: On each sphere in $\mbr^N$ centered at the origin, every polynomial 
      in $\mbr[X_1,\ldots,X_N]$  is equal to some harmonic polynomial. 
      
      \item   (Prop. \ref{P:harmcons}) {\em Let $p\in \mcp_{N,d}$ be homogeneous of degree $d\geq 2$: 
        $$
p=p_0 + p_1X_N + \cdots + p_dX_N^d, 
 $$
where $p_0,\ldots, p_d\in R[X_1,\ldots, X_{N-1}]$. Then $p$
is harmonic if and only if the coefficients $p_{2}, p_{3},\ldots, p_d$ are determined by $p_0$ and $p_{1}$ through the relations}
\begin{equation}
\begin{split}
p_{2k} &=(-1)^{k} \frac{1}{(2k)!}\Delta^kp_0\\
p_{2k+1} &=(-1)^k\frac{1}{(2k+1)!}\Delta^{k}p_{1}
\end{split}
\end{equation} 
{\em for $k \geq 1$. This provides a dimension formula if} $R$ {\em is a field}: 
$$
    \dim_R \mch_{N,d}  =\binom{N+d-2}{N-2}+\binom{N+d-3}{N-2}. 
$$

\item  (Prop. \ref{P:commoneigen}) If $N=2n$ and $i = \sqrt{-1} \in R$ then 
common eigenvectors of  $M_{12}, M_{34}, \ldots, M_{2n-1, 2n}$, and $\mbmn$ in $R[X_1,\ldots,X_N]$ are of the form
$$Y_{a}q,$$
where 
$$Y_a=(X_1+i\varepsilon_1 X_2)^{|a_1|}\cdots (X_{2n-1}+i\varepsilon_{n} X_{2n})^{|a_n|},$$
with $a=(a_1,\ldots, a_n)\in\mbz^n$, $\varepsilon_j \in \{\pm 1\}$ the sign of $a_j$, and $q$ is a polynomial in $X_1^2+X_2^2$,  
$X_3^2+X_4^2,\ldots, X_{2n-1}^2+X_{2n}^2$, with coefficients in $R$, satisfying a certain differential equation; if $N=2$ then $q$ is just an element of $R$. If $N=2n+1$ then similar results are obtained except that $q$ has coefficients in $R[X_N]$.

\item  (Prop. \ref{P:Lc})  A {\em spherical harmonic} is a homogeneous polynomial that is harmonic. A {\em zonal harmonic} is a spherical harmonic that, modulo $\mcz_N(c)$, is a polynomial in $t\cdot X=t_1X_1+\cdots +t_NX_N$, where 
$t=(t_1,\ldots, t_N)\in R^N$. For simplicity suppose $c=1$ and $t\cdot t=1$ here.   {\em For each $q(Y)\in R[Y]$ the polynomial $q(t\cdot X)$ in $R[X_1,\ldots,X_n]/\mcz_N(1)$ is congruent to a spherical harmonic of degree $n$ if and only if $q$ satisfies the differential equation}
$$
  (1-Y^2)q''(Y) -(N-1)Yq'(Y) +n(n+N-2)q(Y) =0.
$$
This equation has, up to scaling by any nonzero element of $R$, a unique solution; this solution is of degree $n$, and 
it is an even polynomial if $n$ is even and it is an odd polynomial if $n$ is odd.

\item  (Prop. \ref{P:sphmean}, Corollary \ref{harmonic0})  We define a {\em  spherical mean} to be an $R$-linear map
$$\lambda:\mcp_N\to R$$
that vanishes on all polynomials of the form $M_{jk}p$ with $p\in\mcp_N$. {\em There exists a unique spherical mean $\lambda_0$ that is equal to $1$ on $(\mbx)^n$ for all integers $n\geq 0$.} Moreover, {\em $\lambda_0(p)=p(0)$ for every harmonic polynomial $p$ in 
$\mcp_N$.}
Analytic counterpart:
$$\lambda_0(p)=\int_{S^{N-1}(1)}p(x_1,\ldots, x_N)\,d\ovs(x),$$
where $\ovs$ is the normalized surface measure on the unit sphere $S^{N-1}(1)$ in $\mbr^N$.

\item  (Equation \ref{E:l0DX23}) {\em With $\lambda_0$ as in} (12) {\em and $p\in\mcp_{N,2n}$  for} $n \geq 1$, 
$$
\lambda_0(p) =  \frac{1}{n!2^n(2n+N-2)(2n+N-4) \ldots (2n+N-2n)} \Delta^{n} p,
$$
 {\em where on the right $\Delta^np$, being of degree $0$, is just an element of $R$}.
 Analytic counterpart:  for a homogeneous polynomial $p \colon \mbr^N \rightarrow \mbr$ of degree $2n \geq 2$, 
\begin{equation*}
    \begin{split}
\int_{S^{N-1}(1)}p(x_1,\ldots, x_N)\,d\ovs(x) &\\
&\hskip -1in= \frac{1}{n!2^n(2n+N-2)\ldots (2n+N-2n)} \Delta^{n} p,
\end{split}
\end{equation*}
where the terms multiplied in the denominator go down by 2 at each step. This is not a standard formula but may be verified by analytic methods. 
\item (Proposition \ref{orth-lambda}) The spherical mean $\lambda_0$ is invariant under rotations: if $p(X)\in R[X_1,\ldots, X_N]$, viewed as a `column vector', where $X=(X_1,\ldots, X_N)$, and $A$ is an $N\times N$ orthogonal matrix with entries in $R$, then $\lambda_0\bigl(p(XA)\bigr)=\lambda_0\bigl(p(X)\bigr)$. We also prove the identity (\ref{E:l0q}), which says that for any homogeneous polynomial $p$ of degree $d$:
\begin{equation}\label{E:l0qintro}
\lambda_0(X_1p)= \frac{1}{N+d-1}\lambda_0\left(\partial_1p\right).
\end{equation}

\item  (Proposition \ref{P:harmonicmeanchar}) If a polynomial $p$ satisfies  the mean-value property
 \begin{equation}\label{E:pmeaninto}
 \lambda_0\bigl(p(X+t)\bigr)=p(t),
 \end{equation}
 where $X=(X_1,\ldots, X_N)$ and $t=(t_1,\ldots, t_N)$ are $N$-tuples of   indeterminates, then $p(X)$ is harmonic.
\end{enumerate}

\subsection{Background and literature}  
A motivation for this work was to develop some of basic results concerning  spherical harmonics in a manner that demonstrates that their validity is purely algebraic and to explain why some of these special functions have 
rational coefficients. 

Algebraic studies of spherical harmonics seem to be largely in the physics literature, with a focus on  atomic physics; these works include Ogawa \cite{Ogawa2017}. Mathematical studies include   Axler and Ramey \cite{Axler1995} and Reznick  \cite{Rez1996}.  Spherical harmonics, originating in work of Legendre and later  Laplace's series expansion of the gravitational potential, are a classical subject; work on them from the 19th and early 20th century include Heine \cite{Heine1861}, Kellogg \cite{Kellogg1929},  Mehler \cite{Mehler1866},   Thomson and Tait \cites{Thom1863, TT1867}, and Whittaker and Watson \cite[section 18.31]{WW1902}. The relationship between special functions and representations of Lie groups is well-known (see, for example Vilenkin and Klimyk \cite{VK1995} and, closer in spirit to our context, Macdonald \cite{McD1995}). 

Our study focuses on polynomials over general rings, typically containing the rationals, rather than  complex-valued special functions that arise as matrix elements for representations of classical Lie groups. Instead of behavior under rotations and using integration over the sphere we use the Lie algebra of the rotation group and we devise purely algebraic counterparts of traditionally analytic notions such as an algebraic counterpart of integration over spheres (Section  \ref{s:sphmea}).

 \section{Polynomials and the action of rotation generators}\label{s:poly}

For a ring $R$, each $p$  in $R[X_1,\ldots, X_N]$ is written as 
 $$
 p= \sum_{{\vec{\jmath}\in\mbz^N}}p_{\vec{\jmath}}X^{\vec{\jmath}} = 
 \sum_{{\vec{\jmath}\in\mbz^N}}p_{\vec{\jmath}}X_1^{j_1}\ldots X_N^{j_N},
 $$
indexed by vectors $\vec{\jmath}= (j_1,\ldots, j_N)\in\mbz^N$, 
with $p_{\vec{\jmath}}=0$ if any component of $\vec{\jmath}$ is negative.
So $p$ is homogeneous of degree $d$ when the monomials $X_1^{j_1}\ldots X_{N}^{j_N}$ 
appearing in $p$ with nonzero coefficients have $j_1+\cdots +j_N =  d$.

  \begin{prop}\label{P:laplace} 
  With $\Delta_N$ being the Laplacian on $R[X_1,\ldots, X_N]$, 
    \begin{equation}\label{magicformula}
({\mbx})\Delta_{N} =  (r\partial_r)^2+(N-2)r\partial_r +{\mbmn}, 
\end{equation}
where   
$$
{\mbx}=X_1^2+\cdots +X_N^2 \ \ \ \text{ and } \ \ \ r\partial_r=\sum_{j=1}^NX_j\partial_j.
$$
\end{prop}
The  identity (\ref{magicformula}) is verified by a straightforward but tedious computation left to the reader. 
It will play a central role for us.

The  Euler operator  $r\partial_r$ on a degree-$n$ homogeneous polynomial $p$ in $R[X_1,\ldots,X_N]$ 
multiplies it by $n$: $(r\partial_r)(p) = np$. Just as the operator $M_{jk}$ arises in the traditional setting as a generator of rotations in the $X_j$-$X_k$ `plane', the Euler operator arises as a generator of scaling each $X_j$ by a common constant.

A simple computation shows that
\begin{equation}\label{E:Mjkmbx}
 M_{jk}(\mbx)= 0,
\end{equation}
and hence, using the derivation property of $M_{jk}$, we also have
\begin{equation}\label{E:Mjka}
 M_{jk}((\mbx)^a)= 0,
\end{equation}
for every integer $a\geq 0$. Consequently,
\begin{equation}\label{E:MdMXdX}
 (\mbmn)((\mbx)^a)= 0.
\end{equation}

We will make repeated use of the fact that $\mbx$, being a monic polynomial in $X_1$, is not a zero divisor. A more general form
of this observation is in  Lemma \ref{L:divide}.

\begin{prop}\label{P:DeltanormX}
With notation as before, for each positive integer $j$ 
\begin{equation}\label{E:DeltapowrX}
\Delta_N((\mbx)^{j})=2j (2j+N-2)({\mbx})^{j-1}. 
\end{equation}

\end{prop}
\begin{proof}
Applying (\ref{magicformula}) to  $(\mbx)^{j}$, and using the annihilation property (\ref{E:MdMXdX}), we have
$$
({\mbx})\Delta_N((\mbx)^{j})= (4j^2+(N-2)2j)(\mbx)^{j},
$$
and then (\ref{E:DeltapowrX}) follows because $\mbx$  is not a zero divisor.
\end{proof}

\subsection{Rotation-invariant polynomials}\label{ss:rotinv}
 
 Geometrically, it is clear that a function on $\mbr^N$  that is invariant under all rotations around the origin 
 is a function of the radial distance from the origin. The following result is an algebraic counterpart of this observation.
 
  \begin{prop}\label{P:Mxyp0N}
  Suppose that $R$ is a ring in which 
  positive integer multiples of the identity $1_R$ are not zero divisors, and $N$ is an integer $\geq 2$.
  A polynomial $p\in R[X_1,\ldots, X_N]$ satisfies $M_{jk}p=0$ for all distinct $j, k\in \{1,\ldots, N\}$ if and only if 
 $$p=q(X_1^2+\cdots+X_N^2),$$
 for some $q(T)\in R[T]$. If $p$ is homogeneous of degree $d$, then $d$ is even and $q(T)=cT^{d/2}$ for some $c\in R$.
 \end{prop}
 
The zero divisor condition on $R$ implies $R$ has characteristic 0 and 
the condition is satisfied if $R$ contains $\mbq$. A ring of characteristic 0 
not fitting the zero divisor condition is $\mbz[x]/(2x)$.

 \begin{proof} 
 As noted in (\ref{E:Mjka}), $M_{jk}((X_1^2 + \cdots + X_N^2)^a) = 0$ for all integers $a \geq 0$, so 
 each polynomial in $R[X_1^2 + \cdots + X_N^2]$ is annihilated by all $M_{jk}$.
 
For the converse, suppose $M_{jk}p = 0$ for all distinct $j$ and $k$. To prove $p$ is a polynomial in 
$\mbx = X_1^2 + \cdots + X_N^2$, we may assume that $p$ is homogeneous since 
each $M_{jk}$ maps homogeneous polynomials to 
homogeneous polynomials of the same degree.

Let $d = \deg p$.  The result is obvious if $d = 0$. 
If $d=1$ then $p$ is of the form $a_1X_1+\cdots +a_NX_N$, so  
$$
M_{jk}p=a_kX_j-a_jX_k.
$$
This being $0$ for all distinct $j, k\in\{1,\ldots, N\}$ implies $p$ is $0$, contradicting  $p$ having degree 1, so 
we can't have $d = 1$.
 
 Let $d \geq 2$. From all $M_{jk} p = 0$ we get $(\mbmn)p = 0$. 
Applying the identity (\ref{magicformula}) to $p$, 
 \begin{equation}
(\mbx){\Delta}p = (d^2 + (N-2)d)p+(\mbmn)p = d(d+N-2)p,
 \end{equation}
 so
 \begin{equation}\label{E:dpqr}
  d(d+N-2)p= (\mbx) F,
  \end{equation}
 where
  $$F={\Delta}p\in R[X_1,\ldots, X_N].$$
 
 Equating the coefficients of $X_1^{j_1}\ldots X_N^{j_N} $ on both sides of (\ref{E:dpqr}), 
 \begin{equation}\label{E:dprN}
 d(d+N-2)p_{j_1,\ldots, j_N}= F_{j_1-2,j_2,\ldots, j_N}+ F_{j_1,j_2-2,\ldots,j_N} + \cdots +F_{j_1,j_2,\ldots, j_N-2}
 \end{equation}
for all $j_1,\ldots,j_N \geq 0$, where a coefficient on the right side of (\ref{E:dprN}) is 0 if one of its indices is negative. 
 We will use induction on $|\vec{\jmath}|' \stackrel{\rm def}{=}j_2+\cdots+j_N$ ($j_1$ is not included in this sum) 
 to show each coefficient of $F_{\vec{j}}=F_{(j_1,\ldots, j_N)}$ is $d(d+N-2)$ times a linear combination of coefficients of $p$.

 If $|\vec{\jmath}|' =0$ then $\vec{\jmath} = (j_1,0,\ldots,0)$, so by 
 (\ref{E:dprN}), 
 $$
 F_{\vec{\jmath}} = F_{j_1,0,\ldots,0} = d(d+N-2)p_{j_1+2, 0,\ldots, 0}.
 $$
 If $|\vec{\jmath}|' =1$ then $\vec{\jmath} = (j_1,0,\ldots,1,\ldots,0)$, so by 
 (\ref{E:dprN}), 
 $$
 F_{\vec{\jmath}} = F_{j_1,0,\ldots,1,\ldots,0} = d(d+N-2)p_{j_1+2, 0,\ldots,1,\ldots, 0}.
 $$

 Suppose, inductively, that for some integer $n \geq 1$, 
 $F_{\vec{\jmath}}$ is $d(d+N-2)$ times a linear combination of 
 coefficients of $p$ whenever $|\vec{\jmath}|'< n$; we have just seen that this is true if $n=1$ or 2. 
 For an $N$-tuple $\vec{\jmath}$ where $|\vec{\jmath}|' = n$, replacing $j_1$ with $j_1+2$ in  (\ref{E:dprN}) gives us 
 $$
F_{\vec{\jmath}}= d(d+N-2)p_{j_1+2,j_2,\ldots, j_N}-\left(F_{j_1+2, j_2-2,\ldots, j_N}+\cdots +F_{j_1+2,j_2,\ldots, j_N-2}\right), 
$$
and the coefficients $F_{\vec{\imath}}$ on the right have $|\vec{\imath}|' = |\vec{\jmath}|' -2 < n$, so 
by induction $F_{\vec{\jmath}}$ is $d(d+N-2)$ times a linear combination of coefficients of $p$ when $|\vec{\jmath}|'=n$.

Thus the polynomial  $F$ itself is $d(d+N-2)$  times a polynomial $F_*\in R[X_1,\ldots, X_N]$ whose coefficients are linear combinations of the coefficients of $p$:
\begin{equation}
F=d(d+N-2)F_*.
\end{equation}
Substituting this into  (\ref{E:dpqr}), we have
\begin{equation}
d(d+N-2)\bigl(p-(\mbx)F_*\bigr) =0.
\end{equation}
Finally we use the zero divisor condition on $R$.  
If any coefficient in $p-(\mbx)F_*$ were nonzero then $d(d+N-2)1_R$ would be a zero divisor in $R$, 
which contradicts the zero divisor condition on $R$ since $d \geq 2$. 
Therefore 
\begin{equation}
p=(\mbx)F_*.
\end{equation}
Since $p$ is homogeneous of degree $d$, the polynomial $F_*$ is homogeneous of degree $d-2$, because   $\mbx$  is not a zero divisor.
Moreover, since $M_{jk}(\mbx) = M_{jk}(X_j^2+X_k^2)  = 0$, by the product rule for the first-order differential operator $M_{jk}$,
for each $g \in R[X_1,\ldots,X_N]$ we have 
   \begin{equation}\label{mjkxx}
   M_{jk}((\mbx)g)=M_{jk}({\mbx})g+(\mbx)M_{jk}g=(\mbx)M_{jk}g.
   \end{equation}
Thus 
\begin{equation}
M_{jk}p= M_{jk}((\mbx)F_*) = (\mbx)M_{jk}F_*.
\end{equation}
The left side is $0$ by hypothesis, so 
$(\mbx)M_{jk}F_* = 0$ for all $M_{jk}$. 
Since $\mbx$ is not a zero divisor in $R[X_1,\ldots, X_N]$, 
\begin{equation}
M_{jk}F_*=0
\end{equation}
for all $M_{jk}$.

To summarize, if $p$ is homogeneous of degree $d \geq 2$ with all $M_{jk}p$ being 0, then 
$p = (\mbx)F_*$ where 
$F_*$ is homogeneous of degree $d -2$ and all $M_{jk}F_*$ are 0.
By induction on the degree, $d-2$ is even and  $F_* = c(\mbx)^{(d-2)/2}$ for some $c\in R$. 
Thus $d$ is even and
$$p=c(\mbx)(\mbx)^{(d-2)/2} = c(X_1^2+\cdots+X_N^2)^{d/2}, 
$$
which completes the proof. 
\end{proof}

\begin{remark}\label{rk1}
If we replace $R$ by $R[Y_1,\ldots, Y_m]$ using new indeterminates $Y_1,\ldots, Y_m$, we get a result for polynomials with `partial' rotational symmetry: if 
$S$ is a subset of $\{1,\ldots, N\}$ containing at least $2$ elements and $M_{jk}p = 0$ for all 
distinct $j, k\in S$ then 
$p = q\left(\sum_{j\in S}X_j^2, \{X_l\}_{l\notin S}\right)$, where $q$ is a polynomial with  coefficients in $R$. 
Proposition \ref{P:Mxyp0N} is the case $S = \{1,\ldots,N\}$. 
\end{remark}

\begin{prop}\label{P:Mquad2} 
Suppose $2$ is not a zero divisor in $R$, and $N$ any positive integer. An $R$-linear derivation on $R[X_1,\ldots, X_N]$ annihilates $\mbx$ if and only if it is an  $R[X_1,\ldots,X_N]$-linear combination of the operators $M_{jk}$.
\end{prop}

\begin{proof} 
We saw in (\ref{E:Mjkmbx})  that $\mbx$ is annihilated by every 
$M_{jk}$, so $\mbx$ is also annihilated by each $R[X_1,\ldots,X_N]$-linear combination of the $M_{jk}$'s. 
We turn now to proving the converse. 

Each $R$-linear derivation on $R[X_1,\ldots,X_N]$ is determined by its effect on $X_1,\ldots,X_N$, so 
 if $L$ is an $R$-linear derivation on $R[X_1,\ldots, X_N]$ then 
\begin{equation}\label{E:Lsum}
L=\sum_{j=1}^Na_j\partial_j
\end{equation}
where $a_j=L(X_j)\in R[X_1,\ldots, X_N]$:  both sides of (\ref{E:Lsum}) 
are $R$-linear derivations that agree on each $X_k$.  

The condition $L(\mbx) = 0$ is equivalent to 
\begin{equation}\label{Lmm}
2\sum_{j=1}^Na_jX_j = 0.
\end{equation} 
Since $2$ is assumed not to be a zero divisor in $R$, (\ref{Lmm}) is equivalent to 
\begin{equation}\label{E:adotX0}
a_1X_1 + \cdots + a_NX_N = 0.
\end{equation}

To deduce from (\ref{E:adotX0}) that $L$ in (\ref{E:Lsum}) is an $R[X_1,\ldots,X_N]$-linear combination of all $M_{jk}$, 
we induct on $N$. We will verify the cases $N=1$ and $N=2$ separately.

For $N=1$, the condition (\ref{E:adotX0}) implies that $a_1=0$, and so $L=0$, which is indeed a linear combination of $M_{11}=0$, trivially.

If $N=2$ then (\ref{E:adotX0}) says 
$$
a_1X_1+a_2X_2=0
$$
in $R[X_1,X_2]$.  Thus $a_1$ is a multiple of $X_2$: $a_1 = bX_2$ where $b \in R[X_1,X_2]$. 
Since $X_2$ is not a zero divisor in $R[X_1,X_2]$,  we get $a_2 = -bX_1$.  Therefore 
$$
L = a_1\partial_1+a_2\partial_2= bX_2\partial_1 - bX_1\partial_2 = bM_{21}, 
$$
which proves the case $N = 2$.

Now assume $N>2$ and the proposition holds for all polynomial rings over $R$ in fewer than $N$ indeterminates.   Write $a_j$ as $a_j^o + b_jX_N$ where the polynomial $a_j^o$ does {\it not} involve $X_N$. Extracting the terms on the left side of (\ref{E:adotX0}) that do not involve $X_N$, we have
\begin{equation}\label{E:adotX1}
a_1^oX_1+\cdots +a_{N-1}^oX_{N-1} = 0.
\end{equation}
Then, inductively, $a_1^o\partial_1+\cdots +a_{N-1}^o\partial_{N-1}$ is 
an $R[X_1,\ldots,X_{N-1}]$-linear 
combination of the operators $M_{jk}$ for distinct $j,k\in\{1,\ldots, N-1\}$. 
Since $M_{kj} = -M_{jk}$, we may take $j < k$: 
\begin{equation}\label{ELaj0N}
a_1^o\partial_1+\cdots +a^o_{N-1}\partial_{N-1}=\sum_{1 \leq j < k \leq N-1} c_{jk}M_{jk}
\end{equation}
for some polynomials $c_{j,k}\in R[X_1,\ldots, X_{N-1}]$. 

Subtracting (\ref{E:adotX1}) from (\ref{E:adotX0}), 
$$
(b_1X_N)X_1 +(b_2X_N)X_2+\cdots +(b_{N-1}X_N)X_{N-1} + a_NX_N=0,
$$
and hence
\begin{equation*}
b_1X_1+b_2X_2+\cdots +b_{N-1}X_{N-1} + a_N=0.
\end{equation*}
Thus the derivation  
$b_1X_N\partial_1+b_2X_N\partial_2+\cdots +b_{N-1}X_N\partial_{N-1} + a_N\partial_N$ is 
$$ 
b_1X_N\partial_1+b_2X_N\partial_2+\cdots +b_{N-1}X_N\partial_{N-1}  -(b_1X_1+b_2X_2+\cdots +b_{N-1}X_{N-1})\partial_N, 
$$
which equals 
$$
b_1(X_N\partial_1 - X_1\partial_N) + \cdots + b_{N-1}(X_N\partial_{N-1} - X_{N-1}\partial_N) = 
\sum_{i=1}^{N-1} b_iM_{Ni}.
$$
Adding this to (\ref{ELaj0N}), we get 
\begin{eqnarray*}
\sum_{j=1}^Na_j\partial_j & = & \sum_{j=1}^{N-1} (a_j^o + b_jX_N)\partial_j  + a_N\partial_N \\
& = & \sum_{j=1}^{N-1} a_j^o\partial_j +  (b_1X_N\partial_1 + \cdots + b_{N-1}X_N\partial_{N-1}  + a_N\partial_N) \\
& = & \sum_{1 \leq j < k \leq N-1} c_{jk}M_{jk}  + \sum_{i=1}^{N-1} b_iM_{Ni}. 
\end{eqnarray*}
\end{proof}

\subsection{Harmonic polynomials}\label{ss:harm}

A polynomial $p$ is {\em harmonic} if $\Delta p=0$; thus the harmonic polynomials in $R[X_1,\ldots,X_N]$ form the kernel of the Laplacian $\Delta_N$.  
Using the algebraic relationship between $\Delta_N$ and the operator of multiplication by $\mbx$, principally coming from 
(\ref{magicformula}), 
we will see how $R[X_1,\ldots,X_N]$ decomposes in terms of these two operators when $R$ contains $\mbq$.

In this section, as before, we use the notation $\mcp_{N,d}$ to denote the $R$-module of all polynomials in $R[X_1,\ldots, X_N]$ that are homogeneous of degree $d$. 
The restriction of $\Delta_N$ to $\mcp_{N,d}$ is denoted $\Delta_{N,d}$, and the $R$-module of 
all harmonic polynomials that are homogeneous of degree $d$ will be denoted
\begin{equation}\label{E:mchNd}
\mch_{N,d} = \ker(\Delta_{N,d})=\ker(\Delta_N|\mcp_{N,d}) .
\end{equation}
Let us note that
\begin{equation}\label{E:Deltagraded}
\Delta_N(\mcp_{N,d})\subset \mcp_{N, d-2}.
\end{equation}
Thus $\Delta_N$ carries homogeneous polynomials into homogeneous polynomials, and so if $p$ is harmonic then each homogeneous component of $p$ is harmonic.

The following simple property of the second order differential operator $\Delta_N$ will be used repeatedly:
\begin{equation}\label{E:Deltprod}
    \Delta_N(pq)=(\Delta_N p)q+2\sum_{j=1}^N(\partial_jp)(\partial_jq)+p\Delta_Nq,
\end{equation}
for all $p, q\in\mcp_N$.
For Proposition \ref{P:kerDelt1}   we will use the following observation. 

\begin{lemma}\label{L:DelraNXdtXmq}
Let $m$ be a non-negative integer and $q\in\mcp_{N,d}$. Then 
\begin{equation}\label{E:DeltapXq}
\Delta_N p = (2m(2m+N-2)+4md)(\mbx)^{m-1}q+(\mbx)^{m}\Delta_N q,
\end{equation}
where $p=(\mbx)^mq$, and the first term on the right side of (\ref{E:DeltapXq}) is taken to be $0$ when $m=0$.
\end{lemma}
\begin{proof}
  Using the product formula (\ref{E:Deltprod}), we have
\begin{equation}
\begin{split}
 \Delta_N p  &=  \Delta_N\bigl((\mbx)^mq\bigr) \\
&=   \Delta_N((\mbx)^{m})q+ 2\sum_{j=1}^Nm(\mbx)^{m-1}2X_j\partial_jq+(\mbx)^{m}\Delta_N q.
\end{split}
\end{equation}  
The first term is 
$2m(2m+N-2)(\mbx)^{m-1}q$ by Proposition \ref{P:DeltanormX}. 
We have $\sum_{j=1}^N X_j\partial_jq = dq$ since $q$ is homogeneous, so 
the second term is $4md(\mbx)^{m-1}q$.  Hence we have (\ref{E:DeltapXq}). \end{proof}

Now we can prove that nonzero multiples of $\mbx$ cannot be harmonic.

\begin{prop}\label{P:kerDelt1}  
Suppose that in the ring $R$ positive integer multiples  of $1_R$ are not zero divisors.
The only harmonic multiple of $\mbx$ in $R[X_1,\ldots, X_N]$ is $0$.
\end{prop}

\begin{proof}
Even though $R$ might not be an integral domain, for nonzero $h$ in $R[X_1,\ldots,X_N]$ we have
\begin{equation}\label{degXh}
\deg((\mbx)h) = 2+\deg h.
\end{equation}
Indeed, this is reduced to the case of homogeneous $h$ by writing $h$ as a sum of homogeneous parts,
and when $h$ is homogeneous every product of monomials in $(\mbx)h$ has the same degree. So 
the only way (\ref{degXh}) fails for homogeneous $h$ is if $(\mbx)h = 0$, which is 
impossible by thinking of $\mbx$ and $h$ as polynomials in $X_N$ since 
$\mbx$ is monic in $X_N$.

From (\ref{degXh}) we have $\deg((\mbx)^mh) = 2m +\deg h$ for all $m \geq 0$, 
so a nonzero polynomial in $R[X_1,\ldots,X_N]$ isn't arbitrarily highly divisible by $\mbx$.

Let $p$ be a harmonic multiple of $\mbx$ in $R[X_1,\ldots, X_N]$. To prove $p = 0$ we will show 
$p$ is arbitrarily highly divisible by $\mbx$: if $(\mbx)^{m} \mid p$ for some $m \geq 1$ 
then $(\mbx)^{m+1} \mid p$.

Let us recall that $\Delta_N$ maps homogeneous polynomials to homogeneous polynomials, and we note also
that any homogeneous component of a polynomial multiple of $\mbx$ is also a multiple of $\mbx$ since multiplication by $\mbx$ maps
any homogeneous polynomial to a homogeneous polynomial with degree raised by $2$.
We may then assume $p$ is homogeneous. 
Set $p = (\mbx)^{m}q$ with $m \geq 1$ and $q \in R[X_1,\ldots,X_N]$.  We may assume $q \not= 0$. 
The polynomial $q$ has to be homogeneous, because if it has nonzero homogeneous 
terms of different degrees then so does $p$ by (\ref{degXh}), which contradicts the homogeneity of $p$. 

As seen in (\ref{E:DeltapXq}),
\begin{equation}\label{E:DeltapXq2}
\Delta_N p = (2m(2m+N-2)+4md)(\mbx)^{m-1}q+(\mbx)^{m}\Delta_N q,
\end{equation}
where $q\in\mcp_{N,d}$.

Since $\Delta_Np = 0$ we have
\begin{equation}\label{E:2mdq}
    2m(2m+N-2+2d)q = -({\mbx})\Delta_N q.
    \end{equation}
We want to deduce from this that $(\mbx) \mid q$ in $R[X_1,\ldots,X_N]$, so 
$(\mbx)^{m+1} \mid p$ in $R[X_1,\ldots,X_N]$ and we'd be done.
 
 Since $\mbx$ is monic in $X_N$, we can use the division algorithm by monics in $R[X_1,\ldots,X_{N-1}][X_N]$   
to write $q = (\mbx)g + r$ where $r = 0$ or $\deg_{X_N}(r) \leq 1$. Multiplying both sides by
$2m(2m+N-2+2d)$ and using (\ref{E:2mdq}) we get $\mbx \mid 2m(2m+N-2+2d)r$, which 
for degree reasons implies $2m(2m+N-2+2d)r = 0$, so $r = 0$ from the zero divisor hypothesis about $R$.
\end{proof}

\begin{prop}\label{P:kerDelt2} 
Suppose that in the ring $R$ positive integer multiples  of $1_R$ are not zero divisors. 
For $c \in R$, the only harmonic multiple of $\mbx - c$ in $R[X_1,\ldots, X_N]$ is $0$.
\end{prop}

This result with $R = \mbr$ is standard in analysis but the statement here is purely algebraic, and the proof will be as well. 

\begin{proof} Suppose $p\neq 0$ is harmonic and a multiple of $({\mbx}-c)$. Then we can write $p$ as
$$p=({\mbx}-c)q,$$
with $q \not= 0$.  Write $q$ as a sum of homogeneous terms:
$$q=q_0+q_1+\cdots +q_m,$$
with each $q_k\in R[X_1,\ldots, X_N]$   homogeneous of degree $k$, and $q_m\neq 0$. 
Therefore the highest-degree homogeneous term in $p$ is $(\mbx)q_m$, and 
$(\mbx)q_m$ is nonzero since $q_m \not= 0$ and $\mbx$ is not a zero divisor 
(for example, it is monic in $X_1$).
From looking at the effect of $\Delta_N$ on homogeneous parts of $(\mbx)q$, we get $\Delta_N((\mbx)q_m) = 0$.
By Proposition  \ref{P:kerDelt1} it follows that $(\mbx)q_m=0$, but we already explained why 
$(\mbx)q_m \not= 0$, so we have a contradiction. 
\end{proof}

\begin{lemma}\label{L:AB} Let $A \colon V \to W$ and $B \colon W\to V$ be linear mappings between finite-dimensional vector spaces over a field. Suppose also that 
$$\ker (A)\cap {\rm Im}(B)=\{0\} \qquad\hbox{and}\qquad \ker B=0.$$
 Then $A$ is surjective, $AB$ is bijective, and
\begin{equation}\label{E:Vd2dec}
V = \ker(A)\oplus  {\rm Im}(B).
\end{equation}
Moreover,
\begin{equation}\label{E:pAB}
v=(I-B(AB)^{-1}A)v + B(AB)^{-1}Av,
\end{equation}
for all $v\in V$, where the first term on the right is in $\ker A$.
\end{lemma}

\begin{proof}  The kernel of $AB \colon W \to W$ is the set of all $w \in W$ for which  $Bw \in \ker A$, which means $Bw \in \ker(A)\cap {\rm Im}(B)=\{0\}$; since $\ker B=0$ it follows then that $\ker(AB)$ is also $0$.  Since $W$ is finite-dimensional over a field, we conclude that $AB$ is bijective, so $A$ is surjective.  In the identity (\ref{E:pAB}), the first term on the right side clearly lies in $\ker(A)$ and the second term is in ${\rm Im}(B)$. Hence $\ker(A)+{\rm Im}(B)=V$. This sum is direct because $\ker(A)\cap {\rm Im}(B)=\{0\}$.
\end{proof}

\begin{prop}\label{P:harmdecomp1}
Suppose that the ring $R$ contains $\mbq$. Then every polynomial $p\in R[X_1,\ldots, X_N]$ can be expressed uniquely in the form
\begin{equation}\label{E:pp0p1}
    p=p_0+(\mbx)p_1+\cdots + (\mbx)^{s}p_s,
    \end{equation}
    where $p_0,\ldots, p_s$ are harmonic polynomials. Moreover, if $p$ is homogeneous of degree $d$ then  $p_j$ is $0$ or 
    homogeneous  of degree $d-2j$ for each non-negative integer $j\leq d/2$.
\end{prop}

This result is widely known in the analytic context; for example, Axler and Ramey \cite{Axler1995} (Theorem 1.7 and Corollary 1.8) give an explicit formula for $p_m$. Reznick \cite{Rez1996}(Theorem 4.7) proves a more general expansion and also traces some earlier works on the expansion (\ref{E:pp0p1}). See also Peterson and Sengupta \cite{PolyHigh2018}(equation (3.12)). The earliest reference we could find to this result is to Gauss \cite{gauss} (page 630). Further work was done by Prasad \cite{prasad1930}.

\begin{proof} 
First we treat the case $R = \mbq$. We use
Lemma \ref{L:AB}, with $V =\mcp_{N,j}$, $W = \mcp_{N,j-2}$, the operator  $A$ being $\Delta_{N,j}=\Delta_N|\mcp_{N,j}$ and $B$ being multiplication by ${\mbx}$. We apply repeatedly, to obtain 
\begin{eqnarray}
\mcp_{N,d} & = & \ker\Delta_{N,d} \oplus (\mbx)\mcp_{N, d-2} \label{E:phdec} \\
& = & \ker \Delta_{N,d} \oplus({\mbx})\ker\Delta_{N, d-2}\oplus (\mbx)^2\mcp_{N, d-4}. \nonumber 
\end{eqnarray}
Continuing in this way we get the direct sum decomposition
\begin{equation}\label{E:mcpdec}
\mcp_{N,d}=\mch_{N,d} \oplus ({\mbx})\mch_{N,d-2} \oplus \ldots ,
\end{equation}
and hence  (\ref{E:pp0p1}) for polynomials with rational coefficients.

In particular, each monomial in $\mbq[X_1,\ldots,X_N]$ is the sum of a harmonic 
polynomial and a multiple of $\mbx$ with $\mbq$-coefficients (preserving homogeneity as the proposition describes). Taking 
$R$-linear combinations of monomials, when $R$ contains $\mbq$, shows every element of $R[X_1,\ldots,X_N]$ 
is the sum of a harmonic 
polynomial and a multiple of $\mbx$ (preserving homogeneity).  This decomposition is unique by 
Proposition \ref{P:kerDelt1}, and the degree part follows from (\ref{degXh}). 
\end{proof}

The decomposition (\ref{E:phdec}), or more simply the surjectivity of $\Delta_{N,d}:\mcp_{N,d}\to\mcp_{N,d-2}$ given by Lemma   \ref{L:AB}, provides the following dimension formula for harmonic polynomials in case $R$ is a field:
\begin{eqnarray}\label{E:dimHNd}
\dim_R \mch_{N,d} & = & \dim_R\mcp_{N,d}-\dim_R\mcp_{N,d-2} \nonumber \\
& = & \binom{N+d-1}{N-1}-\binom{N+d-3}{N-1}, \label{dimf}
\end{eqnarray}
where the second formula is based on 
\begin{equation}\label{E:dimharm}
\dim_R\mcp_{N,d}=\binom{N+d-1}{N-1}, 
\end{equation}
which is the number of monomials of degree $d$ in $N$ indeterminates. Formula (\ref{E:dimHNd}) has appeared in the literature before; for example, \cite{harmfunctionbook} (Theorem 5.8).  We will obtain another expression for the dimension of $\mch_{N,d}$ below in (\ref{E:dimmch2}). 

We will see after Proposition \ref{P:harmcons}
that $\mch_{N,d}$ is a free $R$-module when $R$ is a ring 
containing $\mbq$, not necessarily a field, so (\ref{dimf}) is true as a rank formula for $\mch_{N,d}$ as an $R$-module.
However, we can't take $R = \mbz$: some denominators need to be allowed. 
For example, taking $N = d = 2$ in (\ref{E:phdec}) we have $x^2 = \frac{1}{2}(x^2-y^2) + (x,y) \cdot (x,y) \frac{1}{2}$, 
where $\frac{1}{2}(x^2 - y^2)$ is harmonic.  Even though $x^2$ on the left is in $\mbz[x,y]$, 
on the right $\frac{1}{2}(x^2 - y^2)$ is not in $\mbz[x,y]$.

\subsection{Eigenspaces of \texorpdfstring{$\Delta_N$}{DN}}\label{ss:eigenD}

The Laplacian $\Delta_{N,d}$ sends $\mcp_{N,d}$ to $\mcp_{N,d-2}$, so the product  
$(\mbx) \Delta_{N,d}$ sends $\mcp_{N,d}$ to itself. 
From (\ref{E:mcpdec}) we get an eigenspace decomposition for this operator on $\mcp_{N,d}$, as follows.

\begin{prop}\label{P:Deltadecomp}
Suppose that the ring $R$ contains $\mbq$.  The eigenvalues of ${(\mbx)}\Delta_{N,d}$ acting on $\mcp_{N,d}$ 
are $\lambda_m= 2m(N-2+2d-2m)$ 
for integers $m$ where $0 \leq m \leq d/2$, and 
the $\lambda_m$-eigenspace of ${(\mbx)}\Delta_{N,d}$ is $(\mbx)^{m}\mch_{N,d-2m}$.
In particular, when $R$ is a field of characteristic $0$,  the characteristic polynomial of ${(\mbx)}\Delta_{N,d}$ acting on $\mcp_{N,d}$ is
\begin{equation}\label{E:charXD}
\det\bigl(tI- ({\mbx})\Delta_{N,d}\bigr)=\prod_{m=0}^{[d/2]}\left(t-\lambda_m\right)^{\dim_R\mch_{N,d-2m}}.
\end{equation}
\end{prop}

\begin{proof} The case $d=0$ is clear, and so we assume $d\geq 1$.
For $h\in \ker\Delta_{N, d-2m}$, where $0 \leq m \leq d/2$, we have $\deg h = d-2m$. 
Therefore by (\ref{E:DeltapXq}) with $d$ there replaced by $d-2m$, 
\begin{equation}
\Delta_{N,d}\bigl((\mbx)^{m}h\bigr)= 2m\left(2m+N-2+2(d-2m)\right)(\mbx)^{m-1}h.
\end{equation} 
Therefore $(\mbx)\Delta_{N,d}$ acts on $(\mbx)^{m}\mch_{N,d-2m}$ as multiplication by $\lambda_m$.
Each $(\mbx)^{m}\mch_{N,d-2m}$ is an eigenspace for $(\mbx)\Delta_{N,d}$  and these 
are all of its eigenspaces in  $\mcp_{N,d}$ by the direct sum decomposition in 
Proposition \ref{P:harmdecomp1}.
\end{proof}

\subsection{Construction of harmonic polynomials}\label{ss:conshar}
We show that a harmonic polynomial $p(X_1,\ldots, X_N)$, when viewed as a polynomial in $X_N$, is completely determined by the `constant term' that depends only on the indeterminates $X_1,\ldots, X_{N-1}$.

\begin{prop}\label{P:harmcons}
Suppose $R$ is a ring containing $\mbq$, and $N$ and $d$ are non-negative integers with $N\geq 2$.  Pick $p\in R[X_1,\ldots, X_N]$ 
that is homogeneous of degree $d\geq 0$ and write it as a polynomial in $X_N$: 
\begin{equation}\label{E:pp0pd}
p=p_0 + p_1X_N + \cdots +p_{d}X_N^d,
\end{equation}
where $p_0,\ldots, p_d\in R[X_1,\ldots, X_{N-1}]$. Then $p$
is harmonic if and only if the $X_N$-coefficients $p_{2}, p_{3},\ldots, p_d$ are determined by $p_0$ and $p_{1}$ through the relations 
\begin{equation}\label{E:kerDrec}
\begin{split}
p_{2k} &=(-1)^{k} \frac{1}{(2k)!}\Delta^kp_0 \ {\rm for } \ 1 \leq k \leq d/2, \\
p_{2k+1} &=(-1)^k\frac{1}{(2k+1)!}\Delta^{k}p_{1} \ {\rm for } \ 1 \leq k \leq (d-1)/2.
\end{split}
\end{equation}
\end{prop}

\begin{proof} Let us note first that for any polynomial $q\in R[X_1,\ldots, X_{N}]$ that is independent of $X_N$, we have $\Delta_Nq=\Delta_{N-1}q$.

Since $p$ is homogeneous of degree $d$, in (\ref{E:pp0pd}) each nonzero  $p_k$ is homogeneous 
of degree $d-k$. 

Since $p_j$ and $X_N^j$ are polynomials in disjoint sets of indeterminates, using the product formula (\ref{E:Deltprod}), we have:
$$
\Delta(p_jX_N^j) = p_j \Delta(X_N^j)  + (\Delta p_j)X_N^j =  j(j-1)p_jX_N^{j-2} + (\Delta p_j)X_N^j
$$ 
(the middle term $\sum_{i=0}^d 2(\partial_i p_j)(\partial_i X_N^j)$ is 0 since each summand is 0). 
Therefore
\begin{eqnarray}
\Delta p & = & \sum_{j=0}^d \Delta(p_jX_N^j) \nonumber  \\
& = &  \sum_{j=0}^d \left(j(j-1)p_jX_N^{j-2} + (\Delta p_j)X_N^j\right)  \nonumber \\
& = & \sum_{j=0}^d \left((j+2)(j+1)p_{j+2} + (\Delta p_j)\right)X_N^j. \label{pjj}
\end{eqnarray}
We have $\Delta p = 0$ if and only if 
 every coefficient of $X_N$ in (\ref{pjj}) is 0, which is equivalent to 
 $p_{j} = -(\Delta p_{j-2})/(j(j-1))$ for $j \geq 2$. 
This is the same as (\ref{E:kerDrec}).
 \end{proof}

Since $p_0$ and $p_{1}$ determine $p_j$ for $j \geq 2$ in 
(\ref{E:pp0pd}) when $p$ is harmonic, 
and there is no constraint on $p_0$ and $p_1$ other than their degrees (and being homogeneous if they are not 0) the mapping 
\begin{equation}\label{E:mchiso}
\mch_{N,d}\to \mcp_{N-1,d}\oplus \mcp_{N-1,d-1} \ \text{ given by } \ p\mapsto (p_0, p_{1})
\end{equation}
is an $R$-module isomorphism. 
Both $\mcp_{N-1,d}$ and $\mcp_{N-1,d-1}$ are free $R$-modules for all $R$, so 
if $R$ is a ring containing $\mbq$  
the $R$-module $\mch_{N,d}$ is free 
with a basis of cardinality  
\begin{equation}\label{E:dimmch2}
    \rank_R \mch_{N,d}=\binom{N+d-2}{N-2}+\binom{N+d-3}{N-2},
\end{equation}
where for the right side we use the rank formula (\ref{E:dimharm}) (that formula does not need $R$ to be a field). 
The formula in (\ref{E:dimmch2}) agrees with the one in (\ref{dimf})  
by using a standard binomial coefficient identity twice. 

\begin{corollary}
When $R$ contains $\mbq$, there is a basis of $\mch_{N,d}$ with coefficients in $\mbq$. 
\end{corollary}

\begin{proof}
Both $\mcp_{N-1,d}$ and $\mcp_{N-1,d-1}$ have monomial bases with coefficient 1. By 
(\ref{E:kerDrec}), if $p_0$ and $p_1$ have $\mbq$-coefficients then the associated homogeneous 
harmonic polynomial (\ref{E:pp0pd}) has $\mbq$-coefficients. 
\end{proof}

\begin{example}\label{ex2}
Let us look at the case $N=2$: from (\ref{E:dimmch2}), 
when the coefficient ring $R$ contains $\mbq$ the $R$-module 
$\mch_{2,d}$ of homogeneous harmonic polynomials in $x$ and $y$ of degree $d$ has a basis of size 2.
General polynomials in $\mcp_{N-1,d}$ and $\mcp_{N-1,d-1}$ are $p_0 = ax^d$ and $p_1 = bx^{d-1}$, respectively, 
where  $a$ and $b$ are in $R$. 
Using (\ref{E:kerDrec}), 
the homogeneous harmonic polynomial in $x$ and $y$ of degree $d$ with such $p_0$ and $p_1$ is 
$$
a\sum_{2k \leq d} (-1)^k \binom{d}{2k}x^{d-2k}y^{2k} + b\sum_{2k+1 \leq d} (-1)^k \frac{1}{d}\binom{d}{2k+1}x^{d-2k-1}y^{2k+1}.
$$
The first sum is $((x+iy)^d + (x-iy)^d)/2 = \re((x+iy)^d)$ and the second sum is 
$((x+iy)^d - (x-iy)^d)/(2id) = \im((x+iy)^d)/d$, which recovers the classical fact that 
$\re((x+iy)^d)$ and $\im((x+iy)^d)$ are a basis of $\mch_{2,d}$ when $R = \mbr$.
\end{example}

\subsection{Polynomials constant on spheres}\label{ss:polyconstsphere}

If a polynomial $p$ is in the ideal of $R[X_1,\ldots,X_N]$ generated by ${\mbx}-c$, for some $c\in R$, 
then we think of $p$ as being `equal to zero' on the `sphere' given by the equation ${\mbx}=c$.  Similarly, we say that $p$ is `constant', equal to $c'\in R$,  on the sphere given by ${\mbx}-c$ if $p-c'$ is zero on this sphere.

We will apply the harmonic polynomial decomposition in Proposition \ref{P:harmdecomp1} 
to determine polynomials that are `constant on a sphere' using the following two lemmas.

\begin{lemma}\label{L:divide}
    For a commutative ring $R$,  $p\in R[X_1,\ldots, X_N]$ and $a \in R$ that is not a zero divisor, suppose 
    $ap$ is divisible by $\mbx-c$, where $c\in R$. Then $p$ is divisible by $\mbx-c$. 
    \end{lemma}
    
    \begin{proof} 
    Since $\mbx  - c$ is monic as a polynomial in $X_N$ we can write 
    $$
    p = (\mbx - c)q + r
    $$
    where $q, r \in R[X_1,\ldots,X_{N-1}][X_N]$ with 
    $r = 0$ or $\deg_{X_N}(r) \leq 1$ since $\deg_{X_N}(\mbx) = 2$.
    Multiplying both sides by $a$, 
    $$
    ap = (\mbx - c)aq + ar,
    $$
    so from $ap$ being divisible by $\mbx - c$ we get $(\mbx - c) \mid (ar)$ in $R[X_1,\ldots,X_N]$. 
    Write $ar = (\mbx - c)g$. 
    Since $\mbx  - c$ is monic in $X_N$, 
     if $r \not= 0$ then 
    $\deg_{X_N}(ar) = 2 + \deg_{X_N}(g) \geq 2$,  but 
    $\deg_{X_N}(ar) = \deg_{X_N}(r) \leq 1$ since $a$ is not a zero divisor.  This is a contradiction, 
    so $r = 0$ and thus $(\mbx - c) \mid p$.
  %
%
%
    \end{proof}

\begin{lemma}\label{L:coprimediv}
Let $R$ be a commutative ring, and $c_1,\ldots, c_m$ for $m \geq 2$ 
be distinct elements of $R$ such that no difference $c_i - c_j$ for $i \not= j$ 
is a zero-divisor. If $p\in R[X_1,\ldots, X_N]$ is divisible by $\mbx-c_j$ for 
$j = 1,\ldots, m$ then $p$ is divisible by $\prod_{j=1}^m(\mbx -c_j)$. 
\end{lemma}

\begin{proof} 
First we will show that if $\mbx - a$ divides $(\mbx - b)g$ for $a, b \in R$ 
and $a-b$ is not a zero divisor in $R$ then 
$\mbx - a$ divides $g$.  

For some polynomial $h$ we have $(\mbx - a)h = (\mbx - b)g$.
Then
\begin{equation}
bg =(\mbx)(g-h)+ah,
\end{equation}
so
\begin{equation}\label{E:cq12c12}
(b-a)g =(\mbx-a)(g-h).
\end{equation}
By Lemma \ref{L:divide}, $\mbx-a$ divides $g$.

Now assume $p$ is divisible by $\mbx - c_j$ for $j = 1, \ldots, m$ with $c_i - c_j$ not being a zero divisor 
for all $i \not= j$. 
Inductively we can suppose $p$ is divisible by $(\mbx - c_1)\cdots (\mbx - c_{m-1})$, say $p = (\mbx - c_1)\cdots (\mbx - c_{m-1})q$ in $R[X_1,\ldots,X_N]$.  We will show $\mbx - c_m$ divides $q$ and then we'll be done.
Since $\mbx - c_m \mid (\mbx - c_1)(\mbx - c_2)\cdots (\mbx - c_{m-1})q$, we get 
$\mbx - c_m \mid (\mbx - c_2)\cdots (\mbx - c_{m-1})q$ since $c_m - c_1$ is not a zero divisor. Then  
 $\mbx - c_m \mid (\mbx - c_3)\cdots (\mbx - c_{m-1})q$ since $c_m - c_2$ is not a zero divisor. 
 Eventually  we are left with $(\mbx - c_m) \mid q$ and we are done.
\end{proof}

\begin{prop}\label{P:rotinvMjk}
Suppose that $R$ is a ring containing $\mbq$, and $p\in R[X_1,\ldots, X_N]$, where $N\geq 2$.
\begin{itemize}
\item[(i)] If, for some $c \in R$, all $M_{jk}p$ are divisible by $\mbx-c$ 
then $p$ is equal to a constant modulo $\mbx - c$. 

\item[(ii)] If $p$ is nonzero of degree $d$ and for an $m \geq d/2$ 
there are $c_1, \ldots, c_m$ in $R$ 
whose pairwise differences are not zero divisors with each $M_{jk}p$ divisible by all $\mbx-c_i$  
then $p$ is a polynomial in $\mbx$. 

\item[(iii)] If $p$ is nonzero and homogeneous of degree $d$  and there is $c \in R$ that is not a zero divisor such that all 
 $M_{jk}p$ are divisible by $\mbx-c$  then $d$ is even and  $p=a(\mbx)^{d/2}$ where $a\in R$. 
 \end{itemize}
\end{prop}

If $p$ is equal to a `constant'  mod ${\mbx}-c$ (that is, for some $c' \in R$, $p-c'$ belongs to the ideal generated by $\mbx-c$) 
then every $M_{jk}p$ is  $0$ mod ${\mbx} -c$. The proposition provides  statements in the converse direction to this observation. Note that in (iii) the hypothesis is that, for $p$ homogeneous, $M_{jk}p$ are divisible by $\mbx-c$ and the conclusion implies that $M_{jk}p$ are in fact all $0$.

\begin{proof} 
By Proposition \ref{P:harmdecomp1} we can write
\begin{equation}\label{E:Pexpharm}
    p=p_0+(\mbx) p_1+\cdots + ({\mbx})^{s}p_s,
    \end{equation}
    with $p_0,\ldots, p_s$ all harmonic polynomials. By  (\ref{mjkxx}),  
    $M_{jk}((\mbx)^mp_m) = (\mbx)^m M_{jk}(p_m)$ 
    for all $m \geq 0$, so 
    \begin{equation*}
    M_{jk}p=M_{jk}p_0+(\mbx) M_{jk}p_1+\cdots + ({\mbx})^sM_{jk}p_s.
    \end{equation*}
Now let
    \begin{equation}\label{p-stareq}
    p_*=p_0+c p_1+\cdots+c^sp_s.
    \end{equation}
    This is harmonic (since each $p_k$ is) and equal to $p$ mod  $({\mbx}-c)$.
    
    (i)   Since $M_{jk}(\mbx - c) = 0$, from 
    $p \equiv p_* \bmod \mbx  - c$ we get
     $M_{jk}p_* \equiv M_{jk}p \bmod ({\mbx}-c)$, and so, by the hypothesis in (i),  $M_{jk}p_* \equiv 0 \bmod ({\mbx}-c)$.  
     Since $p_*$ is harmonic and $M_{jk}$ commutes with $\Delta$, 
     $M_{jk}p_*$ is harmonic. Thus by Proposition \ref{P:kerDelt2}, $M_{jk}p_*=0$.  This being true for all 
    distinct $j, k\in\{1,\ldots, N\}$, it follows by Proposition  \ref{P:Mxyp0N} that $p_*$ is a polynomial in $\mbx$ with coefficients in $R$.  
    Therefore $p_* - p_*(0)$ is a multiple of $\mbx$.
    Since $p_*$ is harmonic so is $p_* - p_*(0)$, so $p_* - p_*(0) = 0$ by Proposition \ref{P:kerDelt1}: $p_*$ is a constant. Thus $p$ is equal to a constant modulo $\mbx-c$.

 (ii)  
 First we give a proof when $m > d/2$. By Lemma \ref{L:coprimediv}, each $M_{jk}p$ is divisible by $\prod_{i=1}^m(\mbx - c_i)$, so if $M_{jk}(p) \not= 0$ then $\deg(M_{jk}p) \geq 2m$ by (\ref{degXh}). Since $p$ has degree $d$, $M_{jk}(p)$ has degree at most $d$, so $2m \leq d$. Thus if $m > d/2$ it follows that all $M_{jk}(p) = 0$, so $p$ is a polynomial in 
 $\mbx$ by Proposition \ref{P:Mxyp0N}.
  
The argument above does not apply if $m = d/2$, but the following argument does (and also applies to the case $m > d/2$). By (i) we can write 
\begin{equation}\label{pcq}
p=c_0+(\mbx-c_1)q,  
\end{equation}
    where $c_0\in R$ and $q\in R[X_1,\ldots, X_N]$. 
    Then $\deg q = d-2$ by (\ref{degXh}), unless $q=0$.
    If $d = 2$ then $q$ is constant and (\ref{pcq}) shows 
    $p$ is a polynomial in $\mbx$. 
    Suppose that $d>2$ and $m = d/2$. Since $M_{jk}(\mbx) = 0$, applying $M_{jk}$ to both sides of (\ref{pcq}) gives us $$M_{jk}p=(\mbx - c_1)M_{jk}q.$$
   As seen in the proof of Lemma \ref{L:coprimediv} (the first paragraph in the proof), $M_{jk}q$ is a multiple of ${\mbx}-c_i$ for $i = 2,\ldots, m$; note that $m-1 =  (d-2)/2$. Inductively, we conclude that $q$ is a polynomial in $\mbx$, so $p$ is as well. 
    
  (iii)   When $p$ is homogeneous of degree $d$, each $p_j$ in (\ref{E:Pexpharm}) is 0 or is homogeneous of degree $d-2j$ 
  by Proposition \ref{P:harmdecomp1}. It follows then that for each $j$, the term $c^jp_j$ is the homogeneous part of $p_*$ of degree $d-2j$. By the proof of (i), $p_*$ in (\ref{p-stareq}) 
  is constant; hence each $c^jp_j$ is $0$ unless
  $d-2j=0$. Since $c$ is not a zero divisor, if $c^jp_j = 0$ then $p_j = 0$. Since $p$ is not 0, 
  we get $p = p_{j_0}(\mbx)^{j_0}$, with  $d - 2j_0 = 0$.  Therefore $d = 2j_0$ is even and 
  $p = a(\mbx)^{d/2}$ where $a = p_{j_0}$ is in $R$.
  %
    \end{proof}

Next we note a consequence of Proposition \ref{P:rotinvMjk}.
%

\begin{prop}\label{P:rotinvMjk2}
Let $R$ be a ring containing $\mbq$ and $N \geq 2$. 
Suppose that   $p\in R[X_1,\ldots, X_N]$ is such that $M_{12}p$ is in the ideal generated by $X_1^2+X_2^2-c$ for some $c\in R$. Then
\begin{equation*}
p= a +(X_1^2+X_2^2-c)q
\end{equation*}
where $a \in R$ and $q \in R[X_1,\ldots, X_N]$.
\end{prop}

 \begin{proof}
 Since $M_{12}p$ and $M_{21}p = -M_{12}p$ are divisible by $X_1^2+X_2^2 - c$, 
 by Proposition \ref{P:rotinvMjk}(i) 
 with $R$ replaced by $R[X_3,\ldots, X_N]$ we have $p \equiv a \bmod (X_1^2+X_2^2-c)$ for some $a \in R$.
Therefore $p = a + (X_1^2 + X_2^2-c)q$ where $q$ is in $R[X_3,\ldots,X_N][X_1,X_2] = R[X_1,\ldots,X_N]$.                                                                                                                                       
 \end{proof}
 
 \subsection{Polynomials that vanish on a sphere}\label{ss:zerosphere}
 
 We turn now to showing (Proposition \ref{P:pXN2a})  that, under some conditions on the ring $R$, the ideal of  polynomials in $R[X_1,\ldots, X_N] $ that vanish on the `sphere of radius'  $a$  in $R^N$ is the ideal generated by ${\mbx}-a^2$.

For a ring $R$, any positive integer $k$ and any $a\in R$ we define the $k$-dimensional open ball $B_k(a)$ to be the set of all $(c_1,\ldots, c_k)\in R^k$ such that  $$c_1^2+\cdots +c_k^2+y^2=a^2,$$
for some nonzero $y\in R$.  (This is motivated by the geometry of the case $R = \mbr$.)

For the following results we will impose a certain property on $R$ that ensures that open balls $B_N(a)$ contain infinitely many points on each `slice' specified by fixing the first few coordinates. This includes assuming that $B_1(a)$ itself contains infinitely many points. The condition on $R$, with any nonzero $a\in R$, holds if $R$ is the field of algebraic numbers, the reals $\mbr$, or the complexes $\mbc$.

\begin{lemma}\label{L:solve}
Let $R$ be an integral domain  and suppose that $a\in R$ is  such that the open ball $B_1(a)$ is infinite and, for every integer $k\geq 1$  and every $(c_1,\ldots, c_k)\in B_k(a)$, there are infinitely many $c_{k+1}\in R$ for which $(c_1,\ldots, c_{k+1})\in B_{k+1}(a)$.   If $p \in R[X_1,\ldots, X_N]$ is zero on $B_N(a)$   then $p=0$.
\end{lemma}
 
\begin{proof}
The result holds for $N=1$ because a nonzero polynomial in $R[X]$ cannot have infinitely many zeros in $R$: 
 if $p$ is $0$ at distinct points $\alpha_1,\ldots, \alpha_m$ then, using the fact that $R$ is an integral domain and repeatedly applying the division algorithm, we have
$$p=(X-\alpha_1)\ldots(X-\alpha_m)q$$ for some polynomial $q\in R[X]$ and so $m$ can be at most the degree of $p$.  

Now suppose $N>1$.
Write $p$ as a polynomial in $X_N$ with coefficients in $R[X_1,\ldots, X_{N-1}]$:
$$p=\sum_{k=0}^mp_kX_N^k.$$
For each  $(c_1,\ldots, c_{N-1})\in B_{N-1}(a)$ the polynomial
 $p(c_1,\ldots, c_{N-1}, X_N)$ has infinitely many zeros, one for every point $c_N$ for which $(c_1,\ldots, c_N)\in B_N(a)$. Hence each coefficient $p_k$ is zero at each point of $B_{N-1}(a)$. Then by induction each polynomial $p_k$ is $0$, so $p$ is $0$.\end{proof}

\begin{prop}\label{P:pXN2a}
Let $R$ be an integral domain  and suppose that $a\in R$ is nonzero and such that the open ball $B_1(a)$ is infinite and, for every integer $k\geq 1$  and every $(c_1,\ldots, c_k)\in B_k(a)$, there are infinitely many $c_{k+1}\in R$ for which $(c_1,\ldots, c_{k+1})\in B_{k+1}(a)$.  Suppose also that $R$ does not have 
characteristic $2$. Let $p(X_1,\ldots, X_N)\in R[X_1,\ldots, X_N]$   be such that $p(c_1, \ldots, c_N)=0$ whenever $c_1^2+\cdots+c_N^2=a^2$. Then
\begin{equation}\label{E:pXN2a}
p(X_1,\ldots, X_N)=({\mbx}-a^2)q,
\end{equation}
for some $q\in R[X_1,\ldots, X_N]$.
\end{prop}

\begin{proof} For $N=1$, the hypotheses imply that $p(X_1)$ is divisible by $X_1-a$ and $X_1+a$, which are distinct because $2a\neq 0$. Hence (\ref{E:pXN2a}) holds. 

Suppose $N>1$. By the division algorithm by monic polynomials in the ring  of polynomials in $X_N$ with coefficients in $R[X_1,\ldots, X_{N-1}]$ we have
\begin{equation}\label{E:pdival}
\begin{split}
p(X_1,\ldots, X_N)&=\left(X_N^2 \,+ X_1^2+\cdots +X_{N-1}^2 -a^2\right)q\\
&\qquad +r_1(X_1,\ldots, X_{N-1})X_N+r_0(X_1,\ldots, X_{N-1})
\end{split}
\end{equation}
for some polynomials $q\in R[X_1,\ldots, X_N]$, and $ r_0, r_1\in R[X_1,\ldots, X_{N-1}]$.  We will show  that $r_0$ and $r_1$ are zero by using Lemma \ref{L:solve}.

Let $(c_1,\ldots, c_{N-1})\in B_{N-1}(a)$. Then there is a non-zero $t\in R$ such that
$$c_1^2+\cdots +c_{N-1}^2-a^2= -t^2.$$
Evaluating (\ref{E:pdival}) at $(c_1,\ldots, c_{N-1},t)$, and noting that the left side is $0$ by the assumption on $p$, we have
\begin{equation*}
0=r_1(c_1,\ldots, c_{N-1})t +r_0(c_1,\ldots, c_{N-1}).
\end{equation*}
This holds for a nonzero value of $t$ and  also for $-t\neq t$, since $R$ does not have characteristic $2$.  It follows that both $r_1$ and $r_0$ are $0$ when evaluated at $(c_1,\ldots, c_{N-1})$. Then by Lemma \ref{L:solve}, the polynomials $r_0$ and $r_1$ are both $0$.
\end{proof}

\section{Simultaneous eigenvectors for commuting generators}\label{s:decrot}

In this section we study the common eigenvectors in $R[X_1,\ldots,X_N]$  
of the commuting operators $M_{12}$, $M_{34}$, $\ldots$, $M_{2n-1, 2n}$, where 
$n$ is the largest integer for which $2n\leq N$.

Most, but not all, results in this section will use the hypothesis that in $R$, no 
positive integer multiples  of $1_R$ are zero divisors.  Some results will also use the assumption that $i=\sqrt{-1}\in R$. We will repeat these assumptions on $R$ in the statements of the results.

The following result gives the eigenvectors of each operator $M_{jk}$.

\begin{prop}\label{P:M12p}
Let $R$ be a ring containing  $i=\sqrt{-1}$. Then any polynomial $p\in R[X_1, X_2]$ of the form
\begin{equation}\label{E:peigenM12}
p=(X_1+i\varepsilon_1X_2)^{a_1}q(X_1^2+X_2^2),
\end{equation}
where $a_1$ is a non-negative integer, $\varepsilon_1\in\{\pm 1\}$, and  $q(T) \in R[T]$, satisfies
\begin{equation}\label{E:M12plambd}
M_{12}p=i\varepsilon_1a_1p.
\end{equation}
Conversely, suppose that $R$ is an integral  domain with $i = \sqrt{-1} \in R$ and $2$ is invertible in $R$. 
Then any eigenvector $p$ of $M_{12}$ is of the form $(\ref{E:peigenM12})$.  \end{prop}

\begin{proof} Straightforward computation shows that any polynomial $p$ of the form 
(\ref{E:peigenM12}) satisfies the eigenvector equation (\ref{E:M12plambd}) with $\lambda=i \varepsilon_1 a_1$.

Now suppose $R$ is an integral domain containing $i$  in which $2$ is invertible.  Then every $p\in R[X_1, X_2]$ can be expressed uniquely in the form  
\begin{equation}\label{E:pabcd}
p=\sum_{b,c}k_{bc}(X_1+iX_2)^b(X_1-iX_2)^c,
\end{equation}
where $b, c$ run over all non-negative integers and $k_{bc}\in R$. By straightforward computation,
\begin{equation*}
M_{12}p =\sum_{b,c} i(b-c)k_{bc}(X_1+iX_2)^b(X_1-iX_2)^c.
\end{equation*}
Now suppose $p$ is an eigenvector:
$$M_{12}p=\lambda p.$$
Since $R$ is an integral domain, we have
\begin{equation*}
\lambda=i(b-c)
\end{equation*}
whenever $k_{bc}\neq 0$. We focus now only on those $b, c$ for which $k_{bc}\neq 0$. The value
$$a=b-c$$
is the same for all such pairs $(b,c)$.  If $a\geq 0$ then
\begin{equation*}\label{E:abcd1}
(X_1+iX_2)^b(X_1-iX_2)^c = (X_1+iX_2)^{a}(X_1^2+X_2^2)^c,
\end{equation*}
whereas if $a\leq 0$ then
\begin{equation*}\label{E:abcd2}
(X_1+iX_2)^b(X_1-iX_2)^c=(X_1-iX_2)^{-a}(X_1^2+X_2^2)^b.
\end{equation*}
Thus    (\ref{E:pabcd}) is a sum of  terms of the form $k(X_1+ i\varepsilon_1X_2)^{a_1}(X_1^2+X_2^2)^d$, where $a_1=|a|\geq 0$, 
$\varepsilon_1$ is the sign of $a$ (if $a=0$ we can just set $\varepsilon_1=+1$ for definiteness), and $d \geq 0$.  
The eigenvalue for $p$ is $\lambda = ia = i\varepsilon_1a_1$.
 \end{proof}

Now we can readily obtain eigenvectors that are common to the operators $M_{12}, M_{34}, \ldots, M_{2n-1, 2n}$, which are also eigenvectors  of ${\mbmn}$. 

\begin{prop}\label{P:commoneigen} 
Let $R$ be a ring containing $i=\sqrt{-1}$ and let $N$ be an integer $\geq 2$.  
Let $n$ be the largest integer $\leq N/2$. For  $\varepsilon_1,\ldots, \varepsilon_n \in\{\pm 1\}$ and $a_1,\ldots, a_n$  non-negative integers, set
\begin{equation}\label{E:Yea1}
Y_{\varepsilon,a}=  (X_1+i\varepsilon_1X_{2})^{a_1}\ldots  (X_{2n-1}+i\varepsilon_nX_{2n})^{a_n},
\end{equation}
where we use $\varepsilon$ to denote $(\varepsilon_1,\ldots, \varepsilon_n)$, and $a=(a_1,\ldots, a_n)$.
Then
\begin{equation}\label{E:MYea}
M_{2j-1,  2j}(Y_{\varepsilon,a}) =i \varepsilon_ja_jY_{\varepsilon,a},
\end{equation}
for $j\in\{1, \ldots, n \}$. 

Conversely,  suppose $R$ is an integral domain containing $i$ and $2$ is invertible in $R$. Then every $Y\in\mcp_{N}$ that is an eigenvector of $M_{12}, M_{34}$, ..., $M_{2n-1 \, 2n}$   is of the form
\begin{equation}\label{E:YYea}
Y=Y_{\varepsilon,a}q(X_1^2+X_2^2,\ldots, X_{2n-1}^2+X_{2n}^2)
\end{equation}
for some $\varepsilon$ and $a$ as above, and some $q\in R_*[T_1,\ldots, T_n]$, where $R_*=R$ if $2n=N$ and $R_*=R[X_{2n+1}]$ if $2n+1=N$.  In particular, $Y$ is an $R$-multiple of  $Y_{\varepsilon,a}$ if $Y$ is homogeneous of degree $|a|=a_1+\cdots +a_n$. 
Moreover,
\begin{equation*}\label{E:Yea2}
Y_{\varepsilon, a}\in\ker \Delta_{N,|a|},
\end{equation*}
and $Y_{\varepsilon, a}$ is an eigenvector of ${\mbmn}$.
\end{prop}

\begin{proof} 
The identity (\ref{E:MYea}) is readily verified by computation. For the converse we apply Proposition \ref{P:M12p} with $R[X_3,\ldots, X_N]$ in place of $R$ and use induction. This leads to
\begin{equation}
Y=(X_1+i\varepsilon_1X_2)^{a_1}\ldots (X_{2n-1}+i\varepsilon_nX_{2n})^{a_n}q,
\end{equation}
where $q$ is a polynomial in $X_1^2+X_2^2,\ldots, X_{2n-1}^2+X_{2n}^2$, with coefficients in $R_*$.  If $|a|=d$ then  $q$ must be a constant, an element of $R$. 

It is readily verified that $Y_{\varepsilon, a}$ is harmonic. Since it is also homogeneous, it follows by Proposition \ref{P:laplace} that it is an eigenvector of ${\mbmn}$ as well. 
\end{proof}

 In Proposition \ref{P:commoneigen}  we  did not determine any specific form of the polynomial $q$. Next we obtain a precise description of all harmonic polynomials that are common eigenvectors of the operators $M_{12}, \ldots, M_{2n-1, 2n}$, where $n=[N/2]$.

  \begin{prop}\label{P:Harmpq}
 Suppose $R$ is an integral domain containing $i$ and $2$ is invertible in $R$.  Let $N=2n$ be a positive even integer.
  
  Let $p\in R[X_1,\ldots, X_N]$ be a harmonic polynomial, homogeneous of degree $d$, and suppose $p$ is an eigenvector for the operators $M_{12}$, $M_{34}$, \ldots, $M_{2n-1 \, 2n}$. Then there exist  non-negative integers $a_1,\ldots, a_n$, and $\varepsilon_1,\ldots, \varepsilon_n\in\{\pm 1\}$ such that
   \begin{equation}\label{E:M12p}
  M_{12}p=i\varepsilon_1a_1p,\quad M_{34}p=i\varepsilon_2a_2p,\ldots, M_{2n-1 \, 2n}p=i\varepsilon_na_np,
  \end{equation}
  and 
  \begin{equation}\label{E:pqYe}
  p=Y_{\varepsilon, a}q(X_1^2+X_2^2,\ldots, X_{2n-1}^2+X_{2n}^2),
  \end{equation}
  where $q(T_1,\ldots, T_n)\in R[T_1,\ldots, T_n]$ satisfies
  \begin{equation}\label{E:Tq}
   \sum_{j=1}^n(T_j\partial_j^2+(a_j+1)\partial_j)q(T_1,\ldots, T_n)=0,
   \end{equation}
   where $\partial_j=\partial_{T_j}$.
   
   Conversely, if $p\in R[X_1,\ldots, X_N]$ is homogeneous of degree $d$ and satisfies $(\ref{E:pqYe})$ for some non-negative integers $a_1,\ldots, a_n$, some $\varepsilon_1,\ldots, \varepsilon_n \in\{\pm 1\}$, and some $q\in R[T_1,\ldots, T_n]$ satisfying $(\ref{E:Tq})$ then $p\in \ker \Delta_{N,d}$ and $p$ satisfies the eigenvalue relations $(\ref{E:M12p})$.
     \end{prop}
  
\begin{proof} By  lengthy but straightforward computation, using the product formula (\ref{E:Deltprod}), we have
\begin{equation}\label{E:DelQL}
\Delta p= 4Y_{\varepsilon, a}Lq(X_1^2+X_2^2,\ldots, X_{2n-1}^2+X_{2n}^2),
\end{equation}
if $p$ is of the form (\ref{E:pqYe}), and 
\begin{equation}\label{E:defL}
L= \sum_{j=1}^n(T_j\partial_j^2+(a_j+1)\partial_j).
\end{equation}
Thus $p$ is harmonic if and only if $q\in\ker L$. 
\end{proof} 

In the special case $N=2$, we can work out the condition for $q$ to be in $\ker L$, and it shows that $q$ is of degree $0$. Thus 
the form of $p$ given in (\ref{E:pqYe}) reduces to just
\begin{equation*}\label{E:pqYeN2}
    p=q_0(X_1\pm iX_2)^a
\end{equation*}
for $N=2$, $q_0\in R$, and $a\in\{0,1,2,\ldots\}$.

Let us assume for simplicity that $R$ is a field and contains $\mbq$ and $i$. From 
(\ref{E:dimmch2}) the dimension of $\mch_{2,d}$ over $R$ 
is $2$, and so the two elements $(X_1\pm iX_2)^d$ form a basis of $\mch_{2,d}$, as we already saw in 
Example \ref{ex2}.

 The following result is along the lines of Proposition \ref{P:harmcons}. We use the notation $R[T_1,\ldots, T_n]_d$ for the $R$-module
  of all homogeneous polynomials  in $T_1$, \ldots, $T_n$ of degree $d$, with coefficients in $R$. For $d=-1$ we define it to be $\{0\}$.
 
 \begin{prop}\label{P:harmcons2}
Suppose $R$ is a ring containing $\mbq$, and $n$, $d$, and $a_1,\ldots, a_n$ are non-negative integers. Let $L_d$ be the operator
\begin{equation}\label{E:defLd}
L_d= \sum_{j=1}^n(T_j\partial_j^2+(a_j+1)\partial_j):R[T_1,\ldots, T_n]_d \to R[T_1,\ldots, T_n]_{d-1},
\end{equation}
where   $\partial_j=\partial_{T_j}$.  Then  a polynomial $q$ in $R[T_1,\ldots, T_{n}]_d$, written as
\begin{equation*}\label{E:pp0pd2}
q=q_0T_n^d+\cdots +q_{d-1}T_n+q_d,
\end{equation*}
is in $\ker L_d$ if and only if the coefficients $q_0,\ldots, q_{d-1}  \in R[T_1,\ldots, T_{n-1}]$ are related to $q_d\in R[T_1,\ldots, T_{n-1}]$ by
\begin{equation}\label{E:kerDrec2}
q_{d-k}=(-1)^k\frac{1}{k!(1+a_n)\ldots (k+a_n)}L_d^kq_d
\end{equation}
for all integers $k\in \{0, 1,\ldots, d\}$. In particular, the mapping
\begin{equation*}\label{E:kerL}
\ker L_d \to R[T_1,\ldots, T_{n-1}]_d: q\mapsto q_d
\end{equation*}
is an $R$-linear isomorphism.  In the case $n=1$, $\ker L_d=0$ if $d\geq 1$ and $\ker L_0=R$. 
\end{prop}
\begin{proof} Although $L$ is a second-order differential operator it satisfies the product rule of first-order differential operators when applied to polynomials that depend on different sets of indeterminates:
\begin{equation*}\label{E:Lleib}
 \hbox{
$L(fg)=(Lf)g+f(Lg)$ if $f\in R[T_1,\ldots, T_m]$ and $g\in R[T_{m+1},\ldots, T_n]$.}
\end{equation*}
Applying $L$ to $q$, and using  the above property, we obtain
\begin{equation*}
\begin{split}
  L q  
 &=\left[ d(d+a_n)q_0+L q_1\right]T_n^{d-1} + \left[(d-1)(d+a_n-1)q_1 +L q_2\right]T_n^{d-2} \\
 &\qquad + \left[(d-2)(d+a_n-2)q_2 +L q_3\right]T_n^{d-3}  +\cdots \\
 &\qquad\qquad +\left[2(2+a_{n})q_{d-2}+Lq_{d-1} \right] T_n  +(1+a_n)q_{d-1}+Lq_d
\end{split}
\end{equation*}
So $q\in\ker L$ if and only if  
\begin{equation*}
q_{d-k}=(-1)^k\frac{1}{k!(1+a_n)\ldots (k+a_n)}L^kq_d
 \end{equation*}
for $k$ running up from $1$ to $d$.
 \end{proof}
 
 As we noted in (\ref{E:DelQL}),  if a polynomial $p\in R[X_1,\ldots, X_N]$ is of the form
 \begin{equation}\label{E:pqX12a}
 p=q(X_1^2+X_2^2,\ldots, X_{2n-1}^2+X_{2n}^2)\prod_{j=1}^n(X_{2j-1}\pm i X_{2j})^{a_j},
 \end{equation}
 where $q$  is a polynomial in $ R[T_1,\ldots, T_n]$,
 then
 \begin{equation}
 \Delta p= 4(Lq)(X_1^2+X_2^2,\ldots, X_{2n-1}^2+X_{2n}^2)\prod_{j=1}^n(X_{2j-1}\pm i X_{2j})^{a_j},
 \end{equation}
 where on the right we have the evaluation of the polynomial $Lq$ with the indeterminates set to $(X_1^2+X_2^2,\ldots, X_{2n-1}^2+X_{2n}^2)$.
 
 Thus the $R$-module of all  harmonic polynomials $p\in \ker\Delta_{N,d}$ that are of the form (\ref{E:pqX12a}) is  isomorphic to $\mcp_{n, (d-|a|)/2}$. 
 
 Next we have a counterpart of Proposition \ref{P:Harmpq} for odd values of $N$.
 
 \begin{prop}\label{P:Harmpq2}
 Suppose $R$ is an integral domain containing $i=\sqrt{-1}$ and   $2$ is invertible in $R$.  Let $N=2n+1$ be a positive odd integer, with $n\geq 1$.
  
  Let $p\in R[X_1,\ldots, X_N]$ be a harmonic polynomial, homogeneous of degree $d$, and suppose $p$ is an eigenvector for the operators $M_{12}$, $M_{34}$, \ldots, $M_{2n-1 \, 2n}$. Then there exist  non-negative integers $a_1,\ldots, a_n$, and $\varepsilon_1,\ldots, \varepsilon_n\in\{\pm 1\}$ such that
   \begin{equation}\label{E:M12p2}
  M_{12}p=i\varepsilon_1a_1p,\quad M_{34}p=i\varepsilon_2a_2p,\ldots, M_{2n-1 \,  2n}p=i\varepsilon_na_np,
  \end{equation}
  and 
  \begin{equation}\label{E:pqYe2}
  p=Y_{\varepsilon, a}q(X_1^2+X_2^2,\ldots, X_{2n-1}^2+X_{2n}^2, X_N),
  \end{equation}
  where $q(T_1,\ldots, T_n, Y)\in R[T_1,\ldots, T_n, Y]$ is a homogeneous polynomial of the form
  \begin{equation}\label{E:qTY}
  q(T_1,\ldots, T_n, Y)=\sum_{k=0}^{d-|a|}q_k(T_1,\ldots, T_n)Y^k,
  \end{equation}
  where $q_k(T_1,\ldots, T_n)\in R[T_1,\ldots, T_n]$ is homogeneous  for each $k$, and
    \begin{equation}\label{E:Tq2}
  (k+2)(k+1)q_{k+2} +4Lq_k=0,   \end{equation}
  for all $k\in\{0,1,\ldots, d-|a|-k\}$.  If $d-|a|$ is even then $q_0$ is of degree $(d-|a|)/2$ and $q_1$ is $0$. If $d-|a|$ is odd then $q_0=0$ and $q_1$ has degree $(d-|a|-1)/2$.
  
   Conversely, if $p\in R[X_1,\ldots, X_N]$ is homogeneous of degree $d$ and satisfies $(\ref{E:pqYe2})$ and  $(\ref{E:qTY})$ for some non-negative integers $a_1,\ldots, a_n$, some $\varepsilon_1,\ldots, \varepsilon_n \in\{\pm 1\}$ , and some $q\in R[T_1,\ldots, T_n, Y]$ as in $(\ref{E:Tq2})$, then $p\in \ker \Delta_{N,d}$ and $p$ satisfies the eigenvalue relations $(\ref{E:M12p2})$.
     \end{prop}
     
     \begin{proof} The argument is similar to the proof of  Proposition \ref{P:Harmpq} except that in place of (\ref{E:DelQL}) we use
     \begin{equation}\label{E:DelQL2}
      \Delta\Bigl( Y_{\varepsilon, a} q(X_1^2+X_2^2,\ldots, X_{2n-1}^2+X_{2n}^2, X_N)\Bigr)= Y_{\varepsilon, a} \sum_{k=0}^{d-|a|} s_kX_N^k,
      \end{equation}
      where
     \begin{equation} 
  s_k=
   4Lq_k(X_1^2+X_2^2,\ldots, X_{2n-1}^2+X_{2n}^2) + (k+2)(k+1)q_{k+2}.     \end{equation}

     \end{proof}

 \section{Spherical means}\label{s:sphmea}
 
 In this section we prove an algebraic counterpart (Corollary \ref{harmonic0} below) 
 of the classical analytic result that the mean of a harmonic function over a sphere equals 
 the value of the function at the center of the sphere.
 
 We would like to identify  the integral of a polynomial over a sphere centered at $0$ by using purely algebraic notions. Let $\lambda$ be such an integral, viewed as a linear functional on
  the space of polynomials. Applying $M_{jk}$ to a polynomial $q$  is the effect of taking the derivative of a one-parameter group of  rotations of $q$ in the $X_j$-$X_k$-plane. Thus $\lambda(M_{jk}q)$ should be $0$, since the spherical integral would remain  invariant under rotations. Thus, for any integer $N\geq 2$, we define a {\em spherical mean} to be an $R$-linear map
 \begin{equation*}\label{E:lambd}
 \lambda \colon R[X_1,\ldots, X_N]\to R
 \end{equation*}
 for which
 \begin{equation*}\label{E:kerlambd2}
 \lambda\circ M_{jk}=0\qquad\hbox{for all $j,k\in\{1,\ldots, N\}$ with $j \not= k$.}
 \end{equation*}
 For any polynomials $p$ and $q$, $M_{jk}(pq) = pM_{jk}(q) + qM_{jk}(p)$, so applying $\lambda$ to $M_{jk}(pq)$ shows that 
 \begin{equation*}\label{lambpq}
 \lambda(pM_{jk}q)=-\lambda\left(qM_{jk}p\right).
 \end{equation*}
 
 For $N=1$, when there are no nonzero operators $M_{jk}$, we require instead that $\lambda$ vanish on $X_1^n$ for all odd integers $n\geq 1$ 
 since a sphere centered at 0 in the real line is a pair of points $\{x,-x\}$.

 \subsection{The spherical mean  \texorpdfstring{$\lambda_0$}{L0}}   

 Our next goal is to show that there exists a unique spherical mean $\lambda_0$ on $R[X_1,\ldots, X_N]$ 
 for which $\lambda_0\bigl((\mbx)^{n}\bigr)=1$ for all integers $n\geq 0$, and, moreover, that all spherical means on $R[X_1,\ldots, X_N]$ are $R$-multiples of  $\lambda_0$ when restricted to homogeneous polynomials of any fixed degree.

 We will use the double factorial on nonnegative integers:
\begin{equation}
b!!= \begin{cases}
 b(b-2)\cdots  \cdot 1 \qquad\hbox{if $b$ is odd and  $\geq 1$;}\\
 b(b-2)\cdots  \cdot 2 \qquad\hbox{if $b$ is even and  $\geq 2$;}\\
 1 \qquad \hbox{if $b=0$.}
 \end{cases}
 \end{equation}  
It will also be convenient to set $(-1)!! = 1$, which is a special case of the standard extension of double factorials to negative integers using the recursion $n!! = (n+2)!!/(n+2)$.
 
 \begin{prop}\label{P:sphmean}
 Let $M $ be an $R$-module, where the ring $R$ contains $\mbq$, and let $N$ be an integer $\geq 2$.  
 For an $R$-linear map
 \begin{equation}\label{E:lambdM}
 \lambda \colon R[X_1,\ldots, X_N]\to M
 \end{equation}
 the following two properties are equivalent:
\begin{itemize}
\item for all $j,k\in\{1,\ldots, N\}$ with $j \not=k$, 
 \begin{equation}\label{E:kerlambd}
 \lambda\circ M_{jk}=0,
 \end{equation}
\item for each even $d \geq 0$ there is $v_d \in M$ such that for all monomials $X_1^{a_1}\cdots X_N^{a_N}$, 
 \begin{equation*}\label{E:lambdv}
 \lambda(X_1^{a_1}\cdots X_N^{a_N})= 
 \begin{cases}
 (a_1\!-\!1)!!\cdots (a_N\!-\!1)!!v_{a_1+\cdots+ a_N}, & \text{ if all } a_j \text{ are even}, \\
 0, & \text{ if some } a_j \text{ is odd}.
 \end{cases}
\end{equation*}
 \end{itemize}
 
 Moreover, there exists a unique spherical mean  $\lambda_0$ on $R[X_1,\ldots,X_N]$ whose value on $(\mbx)^{n}$ is $1$ for all integers $n\geq 0$. 
 \end{prop}
 Let us note that the last paragraph above holds even for $N=1$: we set $\lambda_0(X_1^k)$ equal to $0$ for odd $k$ and equal to $1$ for even $k$.
 We also note that the value $\lambda(X_1^{a_1}\cdots X_N^{a_N})$ is invariant under permutations of the indeterminates $X_1,\ldots, X_N$, and so, using $\lambda_0(\mbx)=1$, it follows that
 \begin{equation}\label{E:lambda0Xj}
 \lambda_0(X_j^2)=\frac{1}{N},
 \end{equation}    
 for all $j\in\{1,\ldots, N\}$.

 \begin{proof}[Proof of Proposition \ref{P:sphmean}] 
 First we assume (\ref{E:kerlambd}) holds and derive the formula for $\lambda$ on monomials. 
 Let $X$ and $Y$ be two distinct indeterminates among $X_1,\ldots, X_N$, and let $q$ be a polynomial in the other indeterminates. Then for any integers $a\geq 1$ and $b\geq 0$, we have
 \begin{equation}\label{E:lambXY}
 (X\partial_Y-Y\partial_X)\left(X^{a-1}Y^{b+1}q\right)=  (b+1)X^aY^bq - (a-1)X^{a-2} Y^{b+2}q,
 \end{equation} 
 where the second term on the right is defined to be $0$ if $a=1$.
 Applying $\lambda$ to both sides, we get by (\ref{E:kerlambd}) 
 \begin{equation}\label{E:bp1la}
 0=(b+1)\lambda\left(X^aY^bq\right)-(a-1)\lambda\left(X^{a-2}Y^{b+2}q\right).
 \end{equation}
 Taking $a=1$ we see that $\lambda(Xr)=0$ for all polynomials $r$ in the indeterminates other than $X$. Moreover, if $a$ is an odd positive integer, repeated application of (\ref{E:bp1la}) reduces $\lambda(X^aY^bq)$ to a rational multiple of $\lambda(XY^{b+a-1}q)$, which is $0$. Conversely, if $\lambda(X^aY^bq) = 0$ whenever $a$ is odd and $q$ does not involve $X$ or $Y$ 
 then (\ref{E:bp1la}) holds for  odd $a \geq 1$ since both terms on the right are then $0$. 
 Thus $\lambda(X_1^{a_1}\cdots X_N^{a_N}) = 0$ if
 some $a_j$ is odd.
 
 Now suppose $a$ and $b$ are both even positive integers. Dividing both sides of (\ref{E:bp1la}) by 
 the product of double factorials 
 $(a-1)!!(b+1)!!$, we obtain 
 \begin{equation}\label{E:lambdacommon}
 \frac{1}{(a-1)!!(b-1)!!} \lambda\left(X^aY^bq\right)=\frac{1}{(a-3)!!(b+1)!!} \lambda\left(X^{a-2}Y^{b+2}q\right).
 \end{equation}
It follows 
that $\lambda$ has a common value on all monomials of the form
$$\frac{1}{(a-1)!!(b-1)!!}X^aY^bq,$$
where $a$ and $b$ are even non-negative integers with {\em a fixed value for the sum} $a+b$. Of course, this applies to $X$ and $Y$ being any two of the indeterminates $X_1,\ldots, X_N$. It follows that $\lambda$ has a common value on all monomials of the form
 \begin{equation}\label{E:prodaX}
 \frac{1}{ (a_1-1)!!\cdots (a_N-1)!!}X_1^{a_1}\cdots X_N^{a_N}
   \end{equation}
   for all even integers $a_1,\ldots, a_N\geq 0 $ having a fixed values for the sum $|a|$.
    Conversely, if $\lambda$ has a common value on all monomials of the form
    (\ref{E:prodaX}), where the $a_j$ are all even, then (\ref{E:lambdacommon}) and hence  (\ref{E:bp1la})     hold  for  $X$ and $Y$ running over distinct indeterminates among $X_1,\ldots, X_N$. Then by  (\ref{E:lambXY}), $\lambda(M_{XY}p)$ is $0$ for all polynomials $p$ and hence $\lambda\circ M_{jk}=0$ for all $j,k$.
     
   Thus an $M$-valued spherical mean $\lambda$ is obtained by setting the value of $\lambda$ to be 0 on monomials 
   having an odd exponent, and to have a common value 
   $v_d\in M$ on the monomials in (\ref{E:prodaX}) having all non-negative even exponents and being of total degree $d$.
   
   Focusing now on the case $M=R$, let $\lambda$ be an $R$-valued spherical mean on $R[X_1,\ldots,X_N]$ and set $s_{2n,N} \in R$ to be its common value on the monomials in (\ref{E:prodaX}) of total degree $2n$, where $n \geq 0$. Then 
   \begin{equation}\label{E:lambac}
   \lambda(X_1^{a_1}\cdots X_N^{a_N})= (a_1-1)!!\cdots (a_N-1)!!s_{2n,N},
   \end{equation}
   where the exponents $a_1,\ldots, a_N$ are all even with 
   $a_1+\cdots + a_N = 2n$. Then expanding $(\mbx)^n$ with the multinomial theorem and applying $\lambda$, 
\begin{eqnarray*} 
 \lambda\bigl((\mbx)^n\bigr) & = & \sum_{\substack{b_1,\ldots, b_N \geq 0 \\ b_1+\cdots +b_N = n}}  
 \frac{n!}{b_1!\ldots b_N!}(2b_1-1)!!\ldots (2b_N-1)!!s_{2n,N} \\
 & = & 
 s_{2n,N}\sum_{\substack{b_1,\ldots, b_N \geq 0 \\ b_1+\cdots +b_N = n}}  
 \frac{n!}{b_1!\ldots b_N!}(2b_1-1)!!\ldots (2b_N-1)!!.
\end{eqnarray*}
The summations are really over ordered $N$-tuples $(b_1,\ldots,b_N)$. 
Since the summation is positive, for each $n \geq 0$ there is a unique positive rational number $s_{2n,N}$ that makes $\lambda((\mbx)^n) = 1$. This determines a unique 
$R$-valued spherical mean $\lambda$ on $R[X_1,\ldots,X_N]$, by 
(\ref{E:lambac}), and that $\lambda$ is what we take as $\lambda_0$.
    \end{proof}
   
   Let us note, using (\ref{E:lambac}), that for even 
   $2b_1, \ldots, 2b_N \geq 0$, 
   \begin{equation}\label{E:monint}
   \lambda_0(X_1^{2b_1}\cdots X_N^{2b_N})= (2b_1-1)!!\ldots (2b_N-1)!!s_{2n,N},
   \end{equation}
   where $b_1+\cdots+b_N=n$; the multiplier $s_{2n,N}$, which is chosen to make $\lambda_0\bigl((\mbx)^n\bigr)=1$, is given below in (\ref{E:{s-formula}}).

   \begin{remark}\label{R:rksph}
   We observe, as a consequence of Proposition \ref{P:sphmean}, that the value of a spherical mean on $X_1^{a_1}\cdots X_N^{a_N}$, with each $a_i$ even, is a multiple of $(a_1-1)!!\cdots (a_N-1)!!$, the multiple being determined by the degree $a_1+\cdots+a_N$. Thus, a mapping $\lambda:R[X_1,\ldots,X_N]\to R$ is a spherical mean if and only if it is an $R$-multiple of $\lambda_0$ when restricted to homogeneous polynomials of fixed degree. For this reason we can, without loss of much generality, restrict our attention to $\lambda_0$ rather than on all spherical means.\end{remark}
       
       For a polynomial $p$ with real coefficients, viewed as a function on $\mbr^N$, 
       the value of $\lambda_0(p)$ is the normalized integral of $p$ 
       over the unit sphere in $\mbr^N$ with center at the origin.   The following result reflects the fact that multiplying by a factor of $\mbx$, which would correspond to just multiplying by $1$ over the geometric sphere in $\mbr^N$, does not affect the value of $\lambda_0$ on any polynomial.

       \begin{prop}\label{P:lambdasphere} 
       For a commutative ring $R$ containing $\mbq$,
       \begin{equation}\label{E:lambnormX}
       \lambda_0\bigl((\mbx)p\bigr)=\lambda_0(p)
        \end{equation}
        for all $p\in R[X_1,\ldots, X_N]$.
       \end{prop}

 \begin{proof} The case $N=1$ follows by direct verification. We assume then that $N\geq 2$.
  We will show the  
    $R$-linear map $\mu \colon R[X_1,\ldots,X_N] \rightarrow R$ given by 
    $\mu(p) = \lambda_0((\mbx)p)$ has the properties that uniquely characterize $p \mapsto \lambda_0(p)$.
    
    For each $M_{jk}$, 
    $\mu(M_{jk}p) = \lambda_0((\mbx)M_{jk}(p)) = \lambda_0(M_{jk}((\mbx)p))$ by (\ref{mjkxx}). Thus 
    $\mu(M_{jk}p) = (\lambda_0 \circ M_{jk})((\mbx)p) = 0$, so  $\mu$ is a spherical mean.
    
    We have $\mu((\mbx)^n) = 1$ for all $n \geq 0$ since 
    $$
    \mu((\mbx)^n) = \lambda_0((\mbx)^{n+1}) = 1.
    $$
    Hence $\mu=\lambda_0$.
    \end{proof}
    
As an application of 
Proposition \ref{P:lambdasphere} we next give a formula 
for $s_{2n,N}$ when $n \geq 1$ that is much simpler than the 
reciprocal of the 
summation over $N$-tuples at the end of the proof of 
Proposition \ref{P:sphmean}.

\begin{prop}\label{s-formula}
For each $n \geq 0$,
\begin{equation}\label{E:{s-formula}}
s_{2n,N}=\frac{(N-2)!!}{(N+2n-2)!!}= \frac{1}{(N+2n-2)(N+2n-4)\cdots N}, 
\end{equation}
where successive terms in the denominator drop by $2$.  
\end{prop}

\begin{proof}
 Use Proposition \ref{P:lambdasphere} on $p=X_1^{a_1}\cdots X_N^{a_N}$ where $a_1,\ldots, a_N$ are non-negative even integers with sum $2n$. Since $(\mbx)p$ is homogeneous of degree $2n+2$, writing $\mbx$ as $\sum_{i=1}^n X_i^2$ and using linearity of $\lambda_0$ implies 
 \begin{equation} 
 \begin{split}
\lambda_0((\mbx)p) & =   
((a_1+1)!!(a_2-1)!!\cdots (a_N-1)!!+\cdots \\
& \qquad\qquad   +(a_1-1)!!\cdots (a_{N-1}-1)!!(a_N+1)!!)s_{2n+2,N} \\
&=  (a_1+1+\cdots +a_N+1)\frac{\lambda_0(p)}{s_{2n,N}}s_{2n+2,N} \\
&=  (N+2n)\lambda_0(p)\frac{s_{2n+2,N}}{s_{2n,N}}.
\end{split}
\end{equation}
Division by $s_{2n, N}$ is okay since it is nonzero by the formula for it at the end of the proof of Proposition \ref{P:sphmean}.
Multiplying both sides by $\frac{n!}{(a_1/2)!\cdots (a_N/2)!}$ and summing over all ordered $N$-tuples of nonnegative even integers $a_1, \ldots, a_N$ 
with $a_1 + \cdots + a_N=2n$, we obtain 
\begin{equation}
\lambda_0\bigl((\mbx)^{n+1}\bigr)= (N+2n)\lambda_0\bigl((\mbx)^n\bigr)\frac{s_{2n+2,N}}{s_{2n,N}}.
 \end{equation}
 Thus 
 \begin{equation}\label{E:s2aa}
 s_{2n+2,N}=\frac{1}{N+2n}s_{2n,N}, 
 \end{equation}
which together with the value $s_{0,N}=1$ shows for $n \geq 0$ that 
 \begin{equation}\label{E:s2n}
 s_{2n,N}= \frac{1}{(N+2n-2)(N+2n-4)\cdots N}, 
 \end{equation}
where successive terms in the denominator drop by 2 each time (and the empty product equals $1$).
To get a formula for $s_{2n,N}$ that is more transparently valid at $n = 0$, we
multiply the right side of (\ref{E:s2n}) by $(N-2)!!/(N-2)!!$.  
\end{proof}

Recall from the end of the proof of Proposition 
\ref{P:sphmean} that the condition 
$\lambda_0((\mbx)^n) = 1$ says 
   \begin{equation}\label{E:s2nNform2}
   \sum_{\substack{b_1,\ldots, b_N \geq 0 \\ b_1+\cdots +b_N = n}} \frac{n!}{b_1!\cdots b_N!} (2b_1-1)!!\cdots (2b_N-1)!!s_{2n,N}=1.
   \end{equation}
Combining this with (\ref{E:s2n}) and the formula 
$\binom{2b}{b} = \frac{2^b(2b-1)!!}{b!}$ for $b \geq 0$, 
we get 
   \begin{equation}\label{E:multinform}
   \begin{split}
   \sum_{\substack{b_1,\ldots, b_N \geq 0 \\ b_1+\cdots +b_N = n}} \binom{2b_1}{b_1}\cdots \binom{2b_N}{b_N}&\\
   &\hskip -1in =\frac{2^n}{n!} (N+2n-2)(N+2n-4)\cdots N.
   \end{split}
\end{equation}
For example, when $n=2$ the left side of (\ref{E:multinform}) 
is 
$N\binom{4}{2} + \binom{N}{2}\binom{2}{1}\binom{2}{1} = 2N^2+4N$
and the right side $(2^2/2!)(N+2)N = 2(N+2)N = 2N^2 + 4N$.

The identity (\ref{E:multinform}) can be 
derived in a second way using generating functions: the generating function 
of the left side is $(\sum_{b \geq 0} \binom{2b}{b}X^b)^N$, 
and the generating function of the 
central binomial coefficients $\binom{2b}{b}$ is $(1-4X)^{-1/2}$, whose $N$th power has coefficients given by the right side of (\ref{E:multinform}) using the binomial theorem with exponent $-N/2$. 
A probabilistic interpretation of (\ref{E:multinform}) 
is in \cite{chang-xu}. 
When $N = 2$, (\ref{E:multinform}) is the identity 
$\sum_{b=0}^n \binom{2b}{b}\binom{2(n-b)}{n-b} = 4^n$, 
which has a combinatorial interpretation \cite{sved}. 
The identity (\ref{E:multinform}) is valid at $N=1$ 
if the right side is written as $\frac{2^n(N+2n-2)!!}{n!(N-2)!!}$.

We conclude this discussion with a formula relating $\lambda_0(X_1q)$ to $\lambda_0\left(\partial_1q\right)$.

\begin{prop}\label{P:lambda0par}
  Suppose $R$ is a commutative ring containing $\mbq$, and $N\geq 2$ an integer. Then for any  $p\in R[X_1,\ldots, X_N]$, which is homogeneous of degree $d$, we have
\begin{equation}\label{E:l0q}
\lambda_0(X_1p)= \frac{1}{N+d-1}\lambda_0\left(\partial_1p\right).
\end{equation}
\end{prop}
\begin{proof} We may assume  first that $p$ is a monomial of degree $d$. If any $X_j$, other than $X_1$, appears with odd degree in $p$ then both sides of (\ref{E:l0q}) are $0$, whereas if $X_1$ appears with even degree in $p$ then again both sides of (\ref{E:l0q}) are $0$. Hence we assume that $p=X_1^{2b_1-1}X_2^{2b_2}\cdots X_N^{2b_N}$, where $b_1,\ldots, b_N$ are integers $\geq 0,$ and $b_1\geq 1$. Then
\begin{equation}
    \lambda_0(X_1p)= (2b_1-1)!!\cdots (2b_N-1)!!s_{2n,N},
    \end{equation}
    where
    $$2n=2b_1+\cdots +2b_N=d+1,$$
    and
\begin{equation}
    \lambda_0(\partial_1p)= (2b_1-1)\cdot (2b_1-3)!!(2b_2-1)!!\cdots (2b_N-1)!!s_{2n-2,N}.
    \end{equation}
    Thus
    \begin{equation}
      \lambda_0(X_1p)=\frac{s_{2n,N}}{s_{2n-2, N}}\lambda_0(\partial_1p)   =\frac{1}{N+2n-2}\lambda_0(\partial_1p)
    \end{equation}
    where we used (\ref{E:{s-formula}}). The multiplier on the right here is $1/(N+d-1)$, thereby agreeing with (\ref{E:l0q}).
\end{proof}
 
\subsection{Spherical mean and the Laplacian}\label{ss:sphlapl}
Using Proposition \ref{P:lambdasphere} we can now describe the  normalized spherical mean on homogeneous 
polynomials in terms of iterations of the Laplacian.

\begin{prop}\label{P:l0Del}
Let $R$ be a commutative ring containing $\mbq$.
The normalized spherical mean $\lambda_0$ on $R[X_1,\ldots,X_N]$ can be described on 
homogeneous polynomials $p$ as follows. If $p$ has odd degree then $\lambda_0(p) = 0$. 
For a constant $r$, $\lambda_0(r) = r$. 
If $p$ has even degree $2n \geq 2$ then 
 \begin{equation}\label{E:l0DX23}
 \lambda_0(p)=\frac{1}{n!2^n(2n+N-2)(2n+N-4) \cdots (2n+N-2n)} \Delta^{n} p. 
 \end{equation}
\end{prop}

On the right side of (\ref{E:l0DX23}), $\Delta^np$ is in $R$ since it has degree $0$. 

By Remark \ref{R:rksph}, formula (\ref{E:l0DX23}) holds for a general spherical mean $\lambda$ if a scaling term, depending only on the degree of $p$, is inserted on the right hand side.

\begin{proof} 
If $p$ has odd degree then every monomial in $p$ has an odd exponent, 
so $\lambda_0(p) = 0$. If $p = r$ is constant in $R$, then $\lambda_0(r) = r\lambda_0(1) = r$. 

If $p$ has even degree $d \geq 2$, then apply $\lambda_0$ to the identity (\ref{magicformula})  on $p$. We get 
 \begin{equation}\label{E:lDX2}
 \lambda_0\bigl((\mbx)\Delta p\bigr) = d(d+N-2)\lambda_0(p)
 \end{equation}
since $\lambda_0 \circ M_{jk} = 0$ for all $M_{jk}$. 
 Therefore by Proposition \ref{P:lambdasphere}, 
  \begin{equation}\label{E:l0DX2}
 \lambda_0(p)=\frac{1}{d(d+N-2)} \lambda_0\left(\Delta p\right). 
 \end{equation} 
Writing $d=2n\geq 2$, apply (\ref{E:l0DX2}) $n$ times to obtain (\ref{E:l0DX23}).
\end{proof}

\begin{example}
For $N=4$ and $p=X_1^4X_2^6 =X_1^4X_2^6X_3^0X_4^0$, 
so $n=10/2 = 5$, we have
$$
    \lambda_0(X_1^4X_2^6)=\frac{1}{5!2^5(12 \cdot 10\cdot 8 \cdot 6 \cdot 4)} \Delta^5(X_1^4X_2^6)
$$
The value of $\Delta^5(X_1^4X_2^6)$ is $\binom{5}{2}4!6! = 10 \cdot 4!6!$, so 
$$
\lambda_0(X_1^4X_2^6)=\frac{10 \cdot 4!6!}{5!2^5(12 \cdot 10 \cdot 8 \cdot 6 \cdot 4)}  = \frac{1}{2^9}.
$$
This agrees with the value computed using (\ref{E:lambac}) and   the formula for $s_{2n,N}$ in Proposition \ref{s-formula}:
\begin{equation}
    \lambda_0(X_1^4X_2^6)= 3!!5!! s_{10,4} = 3!! 5!! \frac{1}{12 \cdot 10 \cdot 8\cdot 6\cdot 4} = \frac{1}{2^9}.
\end{equation}
\end{example}
 
 \subsection{The mean-value property and harmonic polynomials}\label{ss:meanVT} In classical analysis it is known that the average value of a harmonic function over a sphere is the value of the function at the center of the sphere. In this subsection we establish purely algebraic results concerning this mean value property and harmonic polynomials. The first result is a direct corollary of Proposition \ref{P:l0Del}.
\begin{corollary}\label{harmonic0}
If $p$ is a harmonic polynomial in $R[X_1,\ldots,X_N]$ 
then the normalized spherical mean of $p$ is the value of $p$ at $0$:
 \begin{equation}\label{eqpL}
\lambda_0(p)=p(0).
\end{equation}
\end{corollary}

\begin{proof} 
Both sides of (\ref{eqpL}) are linear in $p$, so we can assume $p$ is homogeneous.
If $p$ is constant then (\ref{eqpL}) is true by linearity of $\lambda_0$ and the condition $\lambda_0(1) = 1$.
If $p$ is not constant, then by (\ref{E:l0DX2}) we get $\lambda_0(p) = 0$ (the derivation of (\ref{E:l0DX2}) 
works for all nonconstant homogeneous $p$).
\end{proof}

For  $t=(t_1,\ldots, t_N)$, any $N$-tuple of indeterminates such that 
$$(t,X)=(t_1,\ldots, t_N, X_1,\ldots, X_N)$$
is a  $(2N)$-tuple of algebraically independent indeterminates over $R$, we have the normalized spherical mean
\begin{equation}
    \lambda_0:R[t, X_1,\ldots, X_N]\to R[t],
\end{equation}
simply by using the ring $R[t]=R[t_1,\ldots, t_N]$ in place of $R$.  Now consider a harmonic polynomial  $p(Y_1,\ldots, Y_N)\in R[Y_1,\ldots, Y_N]$; then the polynomial
$$p(X_1+t_1,\ldots, X_N+t_N)\in R[t][X_1,\ldots, X_N]$$
is harmonic as a polynomial in $X_1,\ldots, X_N$ with coefficients in $R[t]$. Hence
\begin{equation}\label{E:harmonicmeanXt}
\lambda_0\bigl(p(X+t)\bigr)=p(t),
\end{equation}
by Corollary \ref{harmonic0}.

 The following result is the converse to the observation in (\ref{E:harmonicmeanXt}).  
 
 \begin{prop}\label{P:harmonicmeanchar}
 Let $R$ be a commutative ring containing $\mbq$, and suppose $p(Y)\in R[Y_1,\ldots, Y_N]$, where $Y=(Y_1,\ldots, Y_N)$, satisfies the mean-value property:
 \begin{equation}\label{E:pmean}
 \lambda_0\bigl(p(X+t)\bigr)=p(t)
 \end{equation}
 where $X=(X_1,\ldots, X_N)$ and $t=(t_1,\ldots, t_N)$ are $N$-tuples of   indeterminates. Then $p(X)$ is harmonic.
 
 \end{prop}
 \begin{proof}
 Let us first note that the result holds for $p$ of degree $\leq 1$ because all such polynomials are harmonic. Next, suppose $p$ is of degree $2$. Then, writing
 $$p(Y)=\sum_{m,n=1}^Na_{mn}Y_mY_n+\sum_{m=1}^Na_mY_m+a_0,$$
 we have
 \begin{equation}
     \begin{split}
     \lambda_0\bigl(p(X+t)\bigr) &=\lambda_0\left(\sum_{m,n=1}^Na_{mn}X_mX_n+\sum_{m=1}^Na_mX_m\right)\\
     &\qquad + \sum_{m,n=1}^Na_{mn}\bigl(t_n\lambda_0(X_m)+t_m\lambda_0(X_n)\bigr) +p(t)\\
     &=\frac{1}{N}\sum_{m=1}^Na_{mm} +0 +0 +p(t),
     \end{split}
 \end{equation}
 and so
 $$\sum_{n=1}^Na_{nn}=0,$$
 which implies that $p(Y)$ is harmonic, because $\Delta p(Y)=2\sum_{n=1}^Na_{nn}$.

 We assume then that $p$ has degree $>2$. Let $s=(s_1,\ldots, s_N)$ be another $N$-tuple of indeterminates. Then from (\ref{E:pmean}) we have
 \begin{equation}\label{E:pmean2}
 \lambda_0\bigl(p(X+s+t)\bigr)=p(s+t).
 \end{equation}
 Expanding both sides in powers of the indeterminates $t_j$,   and comparing coefficients of $t^{\underline{j}}=t_1^{j_1}\cdots t_{N}^{j_N}$, where $\underline{j}=(j_1,\ldots, j_N)\in\mbz_{\geq 0}^N$, we have
 \begin{equation}\label{E:pharm3}
\lambda_0\bigl(\partial^{{\underline j}}p(X+s)\bigr)=\partial^{{\underline j}}p(s).
 \end{equation}
 We can focus just on those $\underline{j}$ that have $|{\underline j}|=\sum_{k=1}^Nj_k$ equal to $1$ (that is, we focus on the first partial derivatives of $p$).  Then $\partial^{{\underline j}}p(Y)$ is a polynomial of degree $1$ less than that of $p$, and so the condition (\ref{E:pharm3}) implies, by the induction hypothesis, that $\partial^{{\underline j}}p(Y)$ is harmonic. Hence
 \begin{equation}
\partial^{{\underline j}}\bigl(\Delta p(Y)\bigr)=\Delta \partial^{{\underline j}}p(Y)=0.
 \end{equation}
 Thus $\Delta p(Y)$ is constant, an element of $R$. Writing $p(Y)$ as
 $$p(Y)=p_0+p_1(Y)+\cdots +p_d(Y),$$
 where each $p_k(Y)$ is homogeneous of degree $k$, we observe that $$\Delta p(Y)=\Delta p_2(Y)+\sum_{k=3}^d\Delta p_k(Y),$$
 where $\Delta p_k(Y)$ is homogeneous of degree $k-2$. Since $\Delta p(Y)$ is constant, an element of $R$, it  then follows that $p_k(Y)$ is harmonic for $k\geq 3$. 
 
 The polynomial 
 \begin{equation}
     q(Y)=p_0+p_1(Y)+p_2(Y)=p(Y)-\Sigma_{k=3}^dp_k(Y),
      \end{equation}
      satisfies the mean-value property (\ref{E:pmean}) because it is
       the difference of two polynomials that satisfy (\ref{E:pmean}). 
      Since it is of degree $\leq 2$, $q(Y)$ must then be harmonic.  Hence $p(Y)$, being the sum of $q(Y)$ and $\Sigma_{k=3}^dp_k(Y)$  is also harmonic.  
 \end{proof}

       \subsection{Matrices acting on polynomials}\label{ss:roact} Let ${\rm Matr}_N(R)$ be the $R$-algebra of all $N\times N$ matrices with entries in the commutative ring $R$.  For an $N$-tuple of  indeterminates $X=(X_1,\ldots, X_N)$  and any $A\in {\rm Matr}_N(R)$ we denote by $XA$ the $N$-tuple whose $k$-th component is 
     \begin{equation}\label{E:defXA}
     (XA)_k=\sum_{k=1}^NX_mA_{mk}.
       \end{equation}
        This determines a natural action of  ${\rm Matr}_N(R)$ on $R[X_1,\ldots, X_N]$:
       \begin{equation}\label{E:ApX}
       A[p(X)]=p(XA),
       \end{equation}
       for all $A\in {\rm Matr}_N(R)$ and $p(X)\in R[X_1,\ldots, X_N]$. 
      We note that this is indeed a left action:
             \begin{equation}
       (AB)[p(X)]=p(XAB)= A[Bp(X)],
       \end{equation}
       for any $A, B\in {\rm Matr}_N(R)$. Of interest to us are the orthogonal matrices:
        \begin{equation}\label{E:defON}
        \Ort_N(R)=\{A\in {\rm Matr}_N(R)\,:\, {A^\top}A=I\}.
        \end{equation}
It is a fact that in any commutative ring $R$, the condition ${A^\top}A=I$ implies that $A$ is invertible and hence $A^\top=A^{-1}$ is both a left and a right inverse for $A$.

 \subsection{Rotation-invariance of spherical means}\label{ss:rotinvmean}  We work, as before, with a commutative ring $R$ that contains $\mbq$. A spherical mean $\lambda$ on $R[X_1,\ldots,X_N]$ is an algebraic counterpart to integration over the unit $(N-1)$-sphere centered at the origin in $\mbr^N$ with respect to the uniform measure.  Since that measure on the $(N-1)$-sphere is invariant under rotations of $\mbr^N$ around the origin, integration over that sphere is invariant under an orthogonal change of variables. We want to demonstrate an algebraic analogue: $\lambda_0(p(XA)) = \lambda_0(p(X))$ for all $p$ in $R[X_1,\ldots,X_N]$ and 
 $A$ in $\Ort_N(R) = \{A \in {\rm Matr}_N(R) : A^\top A = I_N\}$. 
 This is Proposition \ref{orth-lambda} below and 
 will be proved in two ways.
 The first proof is based on the commutativity of 
 orthogonal matrices and the Laplacian, and the second proof 
 is based on a formula for the spherical mean of a product of homogeneous linear polynomials in terms of spherical means of product of two homogeneous linear polynomials at a time.

 \begin{lemma}\label{LapOrt-commute}
 For $A \in \Ort_N(R)$,  
 $\Delta^m(p(XA)) = (\Delta^m p)(XA)$ for all 
 $p$ in $R[X_1,\ldots,X_N]$ and  $m \geq 0$.
  \end{lemma}
 
 \begin{proof}
The case $m = 0$ is obvious, and after we establish $m = 1$
the cases $m \geq 2$ follow by induction.

The case $m = 1$ when $R = \mbr$ is a special case of 
the classical 
fact \cite{harmfunctionbook}*{p.~3} that the Laplacian 
on $\mbr^N$ commutes with an orthogonal change of variables.
The proof carries over to $R[X_1,\ldots,X_N]$.  By the chain rule, 
\begin{equation}\label{chain}
\partial_k(p(XA)) = \sum_{i=1}^N (\partial_i p)(XA)a_{ki}
\end{equation} 
for all $p$, where $A = (a_{ij})$.  Therefore
\begin{eqnarray*}
\Delta(p(XA)) & = & \sum_{k=1}^N \partial_k(\partial_k(p(XA))) \\
& = & \sum_{k=1}^N \partial_k\left(\sum_{i=1}^N a_{ki}(\partial_i p)(XA)\right) \text{ by } (\ref{chain}) \\
& = & \sum_{i,k} a_{ki} \sum_{j=1}^N a_{kj}(\partial_j(\partial_i p))(XA)  \text{ by } (\ref{chain}) \\
& = & \sum_{i,j} \left(\sum_{k} a_{ki}a_{kj}\right) (\partial_j(\partial_ip))(XA) \\
& = & \sum_{i} (\partial_i^2 p)(XA) \text{ since } A \in \Ort_N(R) \\
& = & (\Delta p)(XA).
\end{eqnarray*}
\end{proof}

       \begin{lemma}\label{L:pairwise}
       Let $p_1,\ldots, p_{2n}$ be polynomials in $X_1, \ldots, X_N$ that are homogeneous and linear. Then
       \begin{equation}\label{E:pair}
       \lambda_0(p_1\cdots p_{2n})=  N^n  s_{2n,N}\sum_{\pi\in P_{2n}}\prod_{\{a,b\}\in \pi}\lambda_0(p_ap_b),
       \end{equation}
       where $P_{2n}$ is the set of partitions of $\{1,\ldots, 2n\}$ into two-element subsets.
       \end{lemma}

       This result is the spherical counterpart of a feature of the standard Gaussian measure on $\mbr^N$ (for the relationship with Gaussian measure see \cite{PolyHigh2018}*{Theorem 2.1}).
       
       By Remark \ref{R:rksph}, formula (\ref{E:pair}) holds for a general spherical mean $\lambda$ if the right hand side is multiplied by a scaling term in $R$ that depends only on $n$.
       
       \begin{proof} Since both sides of (\ref{E:pair}) are multilinear in $(p_1,\ldots, p_{2n})$, checking the identity is reduced to checking it when each $p_j$ is one of $X_1,\ldots, X_N$. 
       
       If some $X_i$ appears an odd number of times among the $p_j$'s then the  left side is 0 by Proposition \ref{P:sphmean} and the right side is 0 since in each product over a partition $\pi$, $X_i$ appears in some factor $\lambda_0(p_ap_b)$ as $p_a$, with $p_b$ not being $X_i$ (then use Proposition \ref{P:sphmean} again).
       
       Now suppose $2b_1$ of the $p_j$'s are $X_1$ and so on up to $2b_N$ of the $p_j$'s being $X_N$.
       If $j \not= k$ then $\lambda_0(X_jX_k)=0$ by 
       Proposition \ref{P:sphmean} and $\lambda_0(X_j^2)=1/N$ by (\ref{E:lambda0Xj}). Thus a product over a partition $\pi$ on the right side in (\ref{E:pair}) is $0$ unless $p_a=p_b$ for each pair $\{a,b\} \in \pi$, and in this  case the product associated to $\pi$ is $1/N^n$. This value is independent of $\pi$, so the right side of (\ref{E:pair}) is $s_{2n,N}$ times the number of partitions of $\{1,\ldots,2n\}$ into pairs $\{a,b\}$ where $p_a = p_b$, with $\{X_j,X_j\}$ appearing $b_j$ times for $j = 1,\ldots,N$. A set of size $2k$ can be partitioned into pairs in $(2k-1)!!$ ways, so the right side of (\ref{E:pair}) is 
       $$
       s_{2n,N}(2b_1-1)!!\cdots (2b_N-1)!!,
       $$
which is $\lambda_0(X_1^{2b_1}\cdots X_N^{2b_N}) = \lambda_0(p_1\cdots p_{2n})$ by 
(\ref{E:monint}). 
       \end{proof}
 
       Using each of the previous two lemmas we will prove the rotation-invariance of $\lambda_0$.

       \begin{prop}\label{orth-lambda}
    For a commutative ring $R$ containing $\mbq$ and every $N\times N$  matrix $A$ with entries in the ring $R$ that is orthogonal, in the sense that ${A^\top} A = I_N$,  $\lambda_0\bigl(p(XA)\bigr) = \lambda_0\bigl(p(X)\bigr)$ for all $p \in R[X_1,\ldots,X_N]$. 
    \end{prop}
    In view of Remark \ref{R:rksph}, this result holds for any spherical mean $\lambda$.
    \begin{proof} 
  Both sides of the desired formula $\lambda_0\bigl(p(XA)\bigr) = \lambda_0\bigl(p(X)\bigr)$  are linear in $p$, so it suffices to verify this formula when $p$ is a monomial. 
    
    If $p$ is a monomial of odd degree then $\lambda_0(p(X)) = 0$ by Proposition \ref{E:lambdv}. Since $p(XA)$ is a linear combination of odd-degree monomials, also $\lambda_0(p(XA)) = 0$, so the proposition is established in this case.

    Now suppose $p$ is a monomial of even degree. 
    
    \underline{Method 1}. By Proposition \ref{P:l0Del}, when 
    $p$ is homogeneous of even degree $2n$, 
    $\lambda_0(p(X))$ is a multiple of  
    $\Delta^n p \in R$, where the the multiplier 
    is determined by $N$ and $n$. Since $p(XA)$ 
    is homogeneous of the same degree as $p(X)$, $\lambda_0(p(XA))$ is the same scalar multiple of 
    $\Delta^n(p(XA))$, so it suffices
    to show $\Delta^n(p(X)) = \Delta^n(p(XA))$. By 
    Lemma \ref{LapOrt-commute}, 
    $\Delta^n(p(XA))$ equals $(\Delta^n p)(XA)$, which is 
    $(\Delta^n p)(X)$ since $\Delta^n p \in R$.

    \underline{Method 2}. In view of the pair-product formula (\ref{E:pair}) in Lemma \ref{L:pairwise}, it suffices to assume that the polynomial $p(X)$ is a product of two homogeneous linear factors. Thus, without loss of generality, we may assume that $p(X)=X_jX_k$, for some $j,k\in\{1,\ldots, N\}$ that are not necessarily distinct. In this case, we have 
       \begin{equation}
       \begin{split}
           \lambda_0\bigl((XA)_j(XA)_k\bigr)&=\sum_{\ell, m=1}^NA_{\ell j}A_{mk}\lambda_0(X_\ell X_m)\\
           &=\sum_{\ell, m=1}^NA_{\ell j}A_{mk}\frac{\delta_{\ell m}}{N} \\
           &=\sum_{m=1}^NA_{mj}A_{mk}\frac{1}{N}\\
           &=\frac{1}{N}\delta_{jk}\\
           &=\lambda_0(X_jX_k).
           \end{split}
       \end{equation}
       Thus $\lambda_0\bigl(p(XA)\bigr)$ is indeed equal to $\lambda_0\bigl(p(X)\bigr)$ for all $p$.
 \end{proof}

 \section{Spherical harmonics}\label{s:sh}

 In this section $R$ is a ring such that positive integer multiples  of $1_R$ are not zero divisors ({\it e.g.}, $R \supset \mbq$). 
 We follow some of the ideas in \cite{PolyHigh2018}.  
 By a {\em spherical harmonic} we mean a homogeneous harmonic polynomial.  
 
Denote by $\mcz_N(c)$ the ideal generated by ${\mbx}-c$:
\begin{equation}
\mcz_N(c) = \hbox{ideal generated by ${\mbx}-c$ in $R[X_1,\ldots, X_N]$.}
\end{equation}
 Recall 
 $\mcp_N = R[X_1,\ldots,X_N]$

\begin{prop}\label{P:homspher} 
Let $R$ be a commutative ring, $c\in R$ not a zero-divisor, and $\mcz_N(c)$ the ideal in $\mcp_N=R[X_1,\ldots, X_N]$ generated by $\mbx-c$. Then the quotient map $\mcp_N\to \mcp_N/\mcz_N(c)$ is injective when restricted to homogeneous polynomials 
of the same degree. 
\end{prop}

The analytic counterpart of this result is that a homogeneous function on $\mbr^n$ of a specified degree is completely determined by its restriction to a sphere of nonzero radius centered at the origin.

\begin{proof} 
The difference of homogeneous polynomials 
of the same degree is still homogeneous, so 
it suffices to prove that $0$ is the only homogeneous polynomial  in $\mcz_N(c)$. 

Suppose $p\in \mcz_N(c)$ is homogeneous of degree $m$. Then 
\begin{equation}\label{E:pxq}
    p=(\mbx-c)q  =(\mbx)q-cq,
    \end{equation}
    for some $q\in R[X_1,\ldots, X_N]$. If $q\neq 0$ then, since $c$ is not a zero-divisor, the two terms on the right side are of different degrees, contradicting the homogeneity of $p$. Thus $q=0$ and hence $p=0$.
    \end{proof}

\subsection{Commutation relations and \texorpdfstring{$L_c$}{Lc} operators}\label{crsec}
 
The following result and its proof are from \cite{PolyHigh2018}. 
  
    \begin{prop}\label{P:mjkharm} 
    The following commutation relations hold among operators on  
    $R[X_1,\ldots,X_N]$: 
     \begin{equation}\label{E:XnormMjk}
  {\mathcal M}_{\mbx}M_{jk}=M_{jk}{\mathcal M}_{\mbx},
   \end{equation}
   where ${\mathcal M}_{\mbx}$ is the operator that multiplies polynomials by $\mbx$, and 
   \begin{equation}\label{E:MjkDeltaN}
   \Delta_N M_{jk}=M_{jk}\Delta_N, 
   \end{equation}
   so each $M_{jk}$ maps $\mch_{N,d}$ into itself, and 
   \begin{equation}\label{Mnormcommuteeqn}
   M_{jk}({\mbmn})= ({\mbmn})M_{jk}. 
   \end{equation}
       \end{prop}
   \begin{proof}  
   By (\ref{mjkxx}), $M_{jk}((\mbx)p)=(\mbx)M_{jk}p$.
   Thus (\ref{E:XnormMjk}) holds.

   The identity (\ref{Mnormcommuteeqn}) is readily verified by computation,  as already noted in (\ref{E:MjkM2comm}).

   To prove (\ref{E:MjkDeltaN}), it suffices to check both sides are equal on homogeneous polynomials of a fixed degree.
   Consider the expression for 
   ${\mathcal M}_{\mbx}\Delta_N$ in (\ref{magicformula}): $r\partial_r$ acts as a scalar on homogeneous polynomials of a fixed degree and therefore it  commutes with the action of each $M_{jk}$ on those polynomials. Similarly, $\mbmn$ commutes with each $M_{jk}$. 
   Thus (\ref{magicformula}) tells us that 
   ${\mathcal M}_{\mbx}\Delta_N$ commutes with each $M_{jk}$. Since $M_{jk}$ commutes with 
   ${\mathcal M}_{\mbx}$ by (\ref{E:XnormMjk}), 
   it follows that for each homogeneous polynomial $p$, 
   $M_{jk}(\mbx)(\Delta_N(p))) = 
   (\mbx)(\Delta_N(M_{jk}(p)))$, or 
   $$(\mbx)(M_{jk}(\Delta_N(p))) = 
  (\mbx)(\Delta_N(M_{jk}(p))).$$
   Dividing by 
   $\mbx$, we get 
   $M_{jk}(\Delta_N(p)) = \Delta_N(M_{jk}(p))$.
   
   From (\ref{E:MjkDeltaN}), $M_{jk}p$ is harmonic whenever $p$ is harmonic.   \end{proof}

\begin{prop}\label{P:defLc} Suppose $R$ is a ring containing $\mbq$, and $c\in R$. Let $\mch_N=\ker\Delta_N$ denote the submodule of harmonic polynomials in $\mcp_N$. Then the inclusion map $\mch_N\to\mcp_N$ induces  an isomorphism of $R$-modules:
 \begin{equation}
 H_c: \mch_N\to \mcp_N/\mcz_N(c).
 \end{equation}
 Moreover, there is a unique $R$-linear mapping
 \begin{equation}\label{E:defLc}
 L_c: \mcp_N\to \mch_N
 \end{equation}
 which restricts to the identity map on $\mch_N$ and  is $0$ on $\mcz_N(c)$. Furthermore, 
 \begin{itemize}
\item[(i)] $\ker L_c=\mcz_N(c)$;
 \item[(ii)] $H_c\circ L_c$ is the projection map $\mcp_N\to\mcp_N/\mcz_N(c)$;
 \item[(iii)]  $L_c$ commutes with each operator $M_{jk}$.  
 
 \end{itemize}
 \end{prop}
 Item (ii), combined with $H_c(p)=p+\mcz_N(c)$, means that
 \begin{equation}\label{E:Lcharm}
 L_cp\equiv p \,\hbox{mod $\mcz_N(c)$. }
 \end{equation}
 \begin{proof} By Proposition \ref{P:kerDelt2}, if $p\in\mch_N$ is a polynomial multiple of ${\mbx}-c$ then $p=0$; thus $\ker H_c=0$. Next, if $p\in\mcp_N$ then we can write $p$ as
 \begin{equation}\label{E:ppexp}
 p=p_d +({\mbx})p_{d-2} +\cdots +(\mbx)^{n}p_{d-2n},
 \end{equation}
 where each $p_j$ is harmonic,
and so $p$ is equal, modulo ${\mbx}-c$, to the harmonic polynomial $p_d +c p_{d-2} +\cdots +c^{n}p_{d-2n}$. Thus $H_c$ is surjective. With this notation, we can write $p$ in the form
\begin{equation}\label{E:ppcexp}
p=(p_d +c p_{d-2} +\cdots +c^{n}p_{d-2n})+({\mbx}-c)q
\end{equation}
and so $L_c$ is given uniquely by
\begin{equation}\label{E:defLc2}
L_c(p)= p_d +cp_{d-2} +\cdots +c^{n}p_{d-2n}.
\end{equation}
Then 
\begin{equation}
H_c(L_cp)=L_c(p)+\mcz_N(c)=p+\mcz_N(c).
\end{equation}
 
   Applying $M_{jk}$ to the expansion (\ref{E:ppexp})
and using the product rule $M_{jk}(fg)=(M_{jk}f)g+f(M_{jk}g)$ and the fact that $M_{jk}(\mbx)=0$, we have
   \begin{equation}\label{E:Mjkp}
   M_{jk}p=M_{jk}p_d +({\mbx})M_{jk}p_{d-2}+\cdots +(\mbx)^{n}M_{jk}p_{d-2n}.
   \end{equation}
   Here each of the terms $M_{jk}p_{s}$ is harmonic and of the same degree as $p_s$. Hence, (\ref{E:Mjkp}) is the expansion of $M_{jk}p$ in `base' ${\mbx}$, the counterpart of   (\ref{E:ppexp}) for $M_{jk}p$; as seen in Proposition \ref{P:harmdecomp1}, this expansion is unique. Then, by the  definition of $L_c$ in (\ref{E:defLc}), we have: 
   \begin{equation}\label{eqn5432}
   L_cM_{jk}p= M_{jk}p_d +cM_{jk}p_{d-2}+\cdots + c^{n}M_{jk}p_{d-2n}.
   \end{equation}
   Rewriting the right side of (\ref{eqn5432}) with $M_{jk}$ on the left in each term, we get 
   \begin{equation}\label{E:LaMjk}
   L_cM_{jk}p=M_{jk}L_cp.
   \end{equation}

Since each $M_{jk}$ maps  $\mch_{N,d}$ into itself 
(Proposition \ref{P:mjkharm}), for every $d\geq 0$, it follows that $L_c$ commutes with each $M_{jk}$. 
  \end{proof}

 \subsection{Zonal harmonics}\label{ss:zon}
  By a {\em zonal harmonic} we mean a spherical harmonic   that is congruent modulo $\mcz_N(c)$ to a polynomial of the form  
  $q(t_1X_1+\cdots +t_NX_N)$, with $q \in R[T]$ and $t = (t_1,\ldots, t_N) \in R^N$. When $N=3$ this agrees with the terminology in classical analysis. In higher dimensions there are other ways of defining this notion, but for our purposes we choose this more restrictive meaning. 
   
     In the following we use the notation
   \begin{equation}\label{txeqn}
   t\cdot X=t_1X_1+\cdots +t_NX_N\qquad\hbox{and}\qquad t^2=t_1^2+\cdots +t_N^2, 
   \end{equation} 
so $t^2 = t \cdot t$. 

A case of interest is a spherical harmonic that coincides, modulo $\mcz_N(c)$, with a polynomial in the variable $X_1$.   These are spherical harmonics that have rotational symmetry around the axis $(1,0,\ldots, 0)$. Let us verify this. We need to assume that $N\geq 3$. By Remark \ref{rk1}, if $R$ is a ring in which 
positive integer multiples  of $1_R$ are not zero divisors, 
$N\geq 3$, and a polynomial $p\in R[X_1,\ldots, X_N]$ is annihilated by the operators $M_{jk}$ for distinct $j, k \in\{2,\ldots, N\}$ then
\begin{equation}
p=q_*(X_2^2+\cdots+X_N^2),
\end{equation}
where $q_*$ is a polynomial in one variable with coefficients in $R[X_1]$. Thus,
\begin{equation}\label{E:pqX1poly}
p \equiv q(X_1, c-X_1^2) \bmod \mcz_N(c),
\end{equation}
for some polynomial $q$ in two indeterminates, with coefficients in $R$. Thus, for $N>2$, if $p$ is a spherical harmonic that is annihilated by the operators $M_{jk}$ for  all $j,k\in\{2,\ldots, N\}$ with $j \not= k$ then $p$ is a zonal harmonic, corresponding to $t$ being $(1,0,\ldots, 0)$.

\begin{prop}\label{P:Lpq0} Suppose $R$ is a ring containing $\mbq$. Let $L$ denote an arbitrary $R$-linear derivation on $\mcp_N=R[X_1,\ldots, X_N]$ of degree $n\geq -1$, in the sense  that $L$ maps $\mcp_{N,m}$ to $\mcp_{N, m+n}$ for all $m$. Suppose 
$p = t \cdot X$ as in $(\ref{txeqn})$  
with rational coefficients that are not all $0$.  If $q \in \mcp_N$ has the property that $Lp = 0 \Rightarrow Lq = 0$ as $L$ runs over all $R$-linear derivations of $\mcp_N$ as described above, then there is a polynomial $f \in R[T]$ such that $q=f(p)$, and if $q$ is homogeneous then $q$ is an $R$-multiple of a power of $p$. \end{prop}

\begin{proof} By focusing on the homogeneous components of $q$ separately, we may assume that $q$ is homogeneous. Moreover, inductively, we assume that the result holds for all homogeneous polynomials $q$ of lower degree (the degree-$0$ case being automatically true).  

  We   write $L$ as
\begin{equation}\label{E:LLj}
L=\sum_{j=1}^NL_j\partial_j,
\end{equation}
where $L_j=L(X_j)$ is a homogeneous polynomial of degree $n+1$ with coefficients in $R$. Thus the given condition is:
\begin{equation}\label{E:LL}
\hbox{if   $\sum_{j=1}^NL_j\partial_jp=0$  then $\sum_{j=1}^NL_j\partial_jq=0$}
\end{equation}
for all choices of homogeneous  polynomials $L_1,\ldots, L_N$ of degree $n+1$. Pick $k$ such that $t_k \not= 0$. 
For any  $j \not= k$ we take $L_i$ to be $0$ for $i \not= j$ or $k$, $L_j=Y\partial_kp = Yt_k$ and $L_k=-Y\partial_jp = -Yt_j$, where $Y$ is a fixed monomial of degree $n+1$; then the hypothesis of (\ref{E:LL}) is satisfied, so 
\begin{equation}
\partial_kp\partial_jq=\partial_jp\partial_kq.
\end{equation}
This holds for all $j, k\in\{1,\ldots, N\}$. Multiplying both sides by $X_j$ and summing over $j$ we obtain, by homogeneity of $p$ and $q$,
\begin{equation}
Aq=Bp,
\end{equation}
where $A= (\deg q)\partial_kp = (\deg q)t_k \in \mbq^\times$ 
and $B= (\deg p)\partial_kq = \partial_kq$. So

$$q=q_1p,$$
where $q_1$ is also homogeneous, of degree one less than $q$. 
Now if $Lp=0$ for $L$ as in the proposition then we have 
$$Lq = (Lq_1)p+q_1(Lp)=0$$
so $Lq_1=0$.  Since the degree of $q_1$ is one less than the degree of $q$, it follows by the induction hypothesis that $q_1$ is an $R$-multiple of a power of $p$. Hence so is $q$. 

\end{proof}

   \begin{remark}
   If $p=q^2$ then $L(p) = 2qL(q)$ for any first-order differential operator $L$, and thus $L(p) = 0 \Rightarrow L(q) = 0$ even though $q$ is not a polynomial in $p$. This does not contradict the proposition since this $p$ is not of degree 1.
   \end{remark}

For a single-variable polynomial $q(Y) \in R[Y]$, let 
   \begin{equation}\label{qtX}
       q_t=q(t\cdot X) \in R[X_1,\ldots,X_N]
       \end{equation} 
   These polynomials are annihilated by the differential operators that annihilate $t\cdot X$. The converse also holds.

    \begin{prop}\label{P:Lc} Let $R$ be an integral domain of characteristic $0$.  
Let $q(Y)\in R[Y]$ be non-zero  and $t = (t_1,\ldots, t_N) \in R^N$ with coordinates not all $0$, and $N\geq 2$. Let $c\in R$ and $L_c:\mcp_N\to\mch_N$  be the linear map as in $(\ref{E:defLc})$. Then the harmonic polynomial $L_cq_t$ is homogeneous of degree $n$ if and only if $q$ satisfies the differential equation
\begin{equation}\label{E:dtdifc}
  [(t\cdot t)c-Y^2] q''(Y) -(N-1)Yq'(Y) +n(n+N-2)q(Y) =0.
\end{equation}
 \end{prop}
 Let us note that the equation (\ref{E:dtdifc}) involves $t$ only through $t\cdot t$; thus   the same polynomial $q$ works for all $t\in R^N$ having a fixed value of $t\cdot t$. For example, for $t=(1,0,\ldots, 0)$ and $c=1$ the zonal harmonics are of the form $q(X_1)$, with $q$ satisfying (\ref{E:dtdifc}) with $t\cdot t=c=1$.
 
 Suppose $k\geq 1$ is the degree of $q(Y)$. Then, by focusing on the highest degree term on the left side of (\ref{E:dtdifc}), we see that $k=n$.

   \begin{proof}[Proof of Theorem \ref{P:Lc} ]
We have 
   \begin{equation}\label{E:m2qtx}
   \begin{split}
   &{\mbmn} q(t\cdot X) \\
   &= \sum_{\{j,k\}\in P_2(N)} M_{jk}\left((t_kX_j-t_jX_k)q'(t\cdot X)\right)\\
   &= \sum_{\{j,k\}\in P_2(N)} \left[-(t_jX_j+t_kX_k)q'(t_1X_1+\cdots +t_NX_N) \right. \\
   &\qquad\qquad  \left. +  \left(t_kX_j-t_jX_k)^2q''(t_1X_1+\cdots +t_NX_N) \right)\right]\\
   &=-(N-1)(t\cdot X)q'(t\cdot X)+\left[t^2\mbx-(t\cdot X)^2\right]q''(t\cdot X),
   \end{split}
   \end{equation}
where $t^2 = t\cdot t= \sum_{j=1}^N t_j^2$.   
 
   Since $L_c q_t$ is harmonic we have by (\ref{magicformula}) 
   \begin{equation}\label{E:edrqt}
   \begin{split}
   0 & =\mbx\Delta L_cq_t\\
   &=[(r\partial_r)^2+(N-2)r\partial_r +{\mbmn}]L_cq_t \\
   &=[(r\partial_r)^2+(N-2)r\partial_r ]L_cq_t+ L_c{\mbmn}q_t,
   \end{split}
   \end{equation}
   where we used Proposition \ref{P:Lc}(iii) in commuting $L_c$ and $\mbmn$.
   
   Now suppose that the harmonic polynomial $L_c q_t$ is homogeneous of degree $n$. Then the first term in the last line in (\ref{E:edrqt}) is $n(n+N-2)L_cq_t$. Hence, in this case, (\ref{E:edrqt}) is   equivalent to
   \begin{equation}
   n(n+N-2)q_t+{\mbmn}q_t\in\ker L_c.
   \end{equation}
   By Proposition \ref{P:defLc}(i)  then $ n(n+N-2)q_t+{\mbmn}q_t$  is a polynomial multiple of $\mbx-c$.
 The expression in the last line of (\ref{E:m2qtx}) then shows that
 $$ n(n+N-2)q_t -(N-1)(t\cdot X)q'(t\cdot X) +[t^2c-(t\cdot X)^2] q''(t\cdot X)$$
 is a polynomial multiple of $\mbx-c$.  By Lemma \ref{L:qtX} (proved below) this is only possible if
 $$ n(n+N-2)q_t -(N-1)(t\cdot X)q'(t\cdot X) +[t^2c-(t\cdot X)^2] q''(t\cdot X)=0,$$
 so when $L_cq_t$ is homogeneous of degree $n$ this second-order differential equation is satisfied.
 
 Conversely, if the above second-order differential equation is 
 satisfied then using (\ref{E:edrqt}) and then (\ref{E:m2qtx}) we see that
 \begin{equation}\label{missing}
   \begin{split}
 &[(r\partial_r)^2+(N-2)r\partial_r ]L_cq_t \\
 &=-L_c{\mbmn}q_t  \\
 &=L_c\left[ (N-1)(t\cdot X)q'(t\cdot X)-\left[t^2{\mbx}-(t\cdot X)^2\right]q''(t\cdot X)\right]\\
 &=L_c\left[ (N-1)(t\cdot X)q'(t\cdot X)-\left[t^2c-(t\cdot X)^2\right]q''(t\cdot X)\right]\\
 &=n(n+N-2)L_cq_t.
 \end{split} 
 \end{equation}
 Now if we let $p_k$ be the homogeneous degree-$k$ component of $L_cq_t$, and  compare the degree $k$ terms on both sides of (\ref{missing}), we obtain
 $$k(k+N-2)p_k = n(n+N-2)p_k,$$ 
 and so $k(k+N-2) = n(n+N-2)$, because $R$ is an integral domain of characteristic $0$.
 Thus $k = n$ or $k = -(N-2)-n$ and the second option is impossible since $k \geq 0$.
 \end{proof}

 We turn to two results used in the preceding proof.

\begin{lemma}\label{L:qtX} 
Let $R$ be an integral domain of characteristic $0$.   
Let $q(Y)\in R[Y]$ be a polynomial in one variable and suppose 
for some $c \in R$ and $(t_1,\ldots, t_N) \in R^N$, where $N\geq 2$, with at least one $t_j\neq 0$, that 
$q(t\cdot X)\in R[X_1,\ldots, X_N]$ is divisible by $\mbx-c$ in the ring $R[X_1,\ldots, X_N]$.  Then $q=0$.
\end{lemma}
\begin{proof} Suppose
\begin{equation}\label{qtr}
q(t\cdot X)=(\mbx-c)r 
\end{equation} 
for some non-zero $r\in R[X_1,\ldots, X_N]$. 
Set $n = \deg q(Y)$ and let $q_0Y^n$ be the degree-$n$ term in 
$q(Y)$, with $q_0\neq 0$ in $R$. The degree of $r$ in $R[X_1,\ldots,X_N]$ is necessarily $n-2$ since $\mbx - c$ is monic in each $X_i$. 
Focusing on the degree-$n$ terms on both sides, we have 
\begin{equation}\label{E:tXn}
q_0(t\cdot X)^n ={\mbx}r_{n-2},
\end{equation}
where $r_{n-2}$  is the (necessarily nonzero) term of degree $n-2$ within $r$ and the right side of (\ref{E:tXn}) is nonzero. 
Then we have a contradiction from Lemma \ref{L:factor1} (proved below).
\end{proof}

\begin{lemma}\label{L:factor1}
Let $R$ be an integral domain of characteristic $0$, and $t=(t_1,\ldots, t_N)\in R^N$, with at least one nonzero component.
Let $p(Y)\in R[Y]$ be such that $p(t_1X_1+\cdots +t_NX_N)$ is divisible by $X_1^2+\cdots +X_N^2$ in the polynomial ring $R[X_1,\ldots, X_N]$ and $N \geq 2$. Then $p=0$.

 \end{lemma}
 \begin{proof}  Any nonzero polynomial multiple of $Q=X_1^2+\cdots +X_N^2$ is of degree at least $2$ in each $X_j$.  Suppose that
 \begin{equation}\label{E:pQb}
 p(t_1X_1+\cdots +t_NX_N)=Qf,
 \end{equation}
 where  $f\in R[X_1,\ldots, X_N]$. Evaluating at $X_j=0$ for all $j>2$, we can focus on the case of just two indeterminates
 \begin{equation}\label{E:pQf}
 p(t_1X_1+t_2X_2)=(X_1^2+X_2^2)f(X_1,X_2).
 \end{equation}
 Let $n = \deg p(Y)$ and $p_0Y^n$ be the degree-$n$ term in $p(Y)$; 
 so $p_0\in R$ is non-zero and $n\geq 0$. 
 
 We look at the degree-$n$ terms on both sides of (\ref{E:pQf}). 
 On the left side of (\ref{E:pQf}) we get $p_0(t_1X_1+t_2X_2)^n$.
 \begin{equation}
 \begin{split}
 &(t_1X_1+t_2X_2)^n \\
 &=(t_2X_2)^n+n(t_2X_2)^{n-1}t_1X_1 + \sum_{j=1}^{[n/2]} \binom{n}{2j}t_1^{2j}(X_1^2)^{j}(t_2X_2)^{n-2j} \\
 &\qquad+X_1\sum_{j=1}^{[(n-1)/2]} \binom{n}{2j+1}t_1^{2j+1}(X_1^2)^{j}(t_2X_2)^{n-(2j+1)}.
 \end{split}
 \end{equation}
 Upon writing $X_1^2$ as $Q-X_2^2$, we have
 \begin{equation}
 \begin{split}
 &(t_1X_1+t_2X_2)^n \\
  &=(t_2X_2)^n+n(t_2X_2)^{n-1}t_1X_1 + \sum_{j=1}^{[n/2]} \binom{n}{2j}t_1^{2j}(Q-X_2^2)^{j}(t_2X_2)^{n-2j} \\
 &\qquad+X_1\sum_{j=1}^{[(n-1)/2]} \binom{n}{2j+1}t_1^{2j+1}(Q-X_2^2)^{j}(t_2X_2)^{n-(2j+1)}\\
 &=a X_1+b +c Q,
 \end{split}
 \end{equation}
 where $a,b\in R[X_2]$ and $c \in R[X_1,X_2]$; specifically,
 \begin{equation}\label{E:defbn}
 \begin{split}
 a&= \sum_{j=0}^{[(n-1)/2]} \binom{n}{2j+1}(-1)^jt_1^{2j+1}    t_2^{n-(2j+1)} X_2^{n-1}    \\
 b  &= \sum_{j=0}^{[n/2]}\binom{n}{2j}(-1)^jt_1^{2j}t_2^{n-2j}X_2^{n}.
 \end{split} 
 \end{equation}
  Hence,
 \begin{equation}
 (t_1X_1+t_2X_2)^np_0=  p_0aX_1+ p_0b+ cp_0Q.
 \end{equation}
 Since the left side is a multiple of $Q$ as in (\ref{E:pQf}), we have
 \begin{equation}
( f- p_0c)Q=p_0aX_1+p_0 b.
\end{equation}
The left side is either $0$ or of degree $\geq 2$ in $X_1$, whereas the right side is either $0$ or of degree $1$ in $X_1$. Thus   both sides must be $0$. But neither $a$ nor $b$ is $0$, and so $p_0=0$, a contradiction.
\end{proof}

The solutions to equation (\ref{E:dtdifc}) are traditionally understood in terms of Gegenbauer polynomials. We conclude with the following observation.

 \begin{prop}\label{P:solGegen} Suppose $R$ is an integral domain containing $\mbq$,  $\alpha\in R$,   and $q(Y)\in R[Y]$ a nonzero polynomial that satisfies the differential equation
 \begin{equation}\label{E:dtdifc2}
  [\alpha-Y^2] q''(Y) -(N-1)Yq'(Y) +n(n+N-2)q(Y) =0,
\end{equation}
where $n\geq 0$ and $N\geq 3$ are integers.
Then the degree of $q$ is $n$.  If $\alpha=0$ then $q(Y)=Y^n$ is, up to multiplication by any element of $R$, the unique solution to (\ref{E:dtdifc2}). If $\alpha\neq 0$   then a   solution $q(Y)$, unique up to scaling by any element of $R$,  to $(\ref{E:dtdifc2})$ exists, with coefficients in $R[\alpha^{-1}]$, and is an even polynomial if $n$ is even and an odd polynomial of $n$ is odd.
 \end{prop}
 \begin{proof} Substituting
 $$q(Y)=\sum_{k=0}^mq_kY^k$$
 into (\ref{E:dtdifc2}), and examining the coefficient of $Y^k$,   (\ref{E:dtdifc2}) is equivalent to
 \begin{equation}\label{E:diffeqqk}
 \begin{split}
 {\alpha} q_{k+2} &= \frac{k(k-1)+k(N-1)-n(n+N-2)}{(k+2)(k+1)}q_k\\
 &=\frac{(k-n)(k+n+N-2)}{(k+2)(k+1)}q_k,
 \end{split}
 \end{equation}
 holding for all $k\in\{0,1,2,\ldots\}$.  The condition $N\geq 3$ implies that $k+n+N-2\neq 0$ since $k$ and $n$ are non-negative integers.
 
 If $\alpha=0$ then $q_k=0$ unless $k=n$.  Now suppose $\alpha\neq 0$.  If $n$ is even and $q_1$ were nonzero then (\ref{E:diffeqqk}) would imply that $q_k$ is nonzero for all odd $k$; this would contradict the fact that $q(Y)$ is a polynomial. Thus $q(Y)$ is an even polynomial. Moreover, (\ref{E:diffeqqk}) also implies that $q_m=0$ for all even $m>n$. A similar argument works for $n$ odd.
 \end{proof}

{\bf Acknowledgments}. We thank   Brian Hall, Arthur Parzygnat, Amy Peterson, and Tom Roby for discussions on this topic. Special thanks to Darij Grinberg for detailed and very useful comments and corrections to an earlier version of this paper.

 \bibliographystyle{amsplain}
 \bibliography{Harmony}{}
\bibliographystyle{plain}

\end{document}